\documentclass{amsart}
\usepackage{a4}
\usepackage{amssymb,amsmath,amsfonts,amsthm}
\usepackage[T1]{fontenc}

\usepackage{color}

\usepackage{tikz}
\usetikzlibrary{arrows}

\newtheorem{theorem}{Theorem}[section]
\newtheorem{lemma}{Lemma}[section]
\newtheorem{corollary}{Corollary}[section]
\newtheorem{proposition}{Proposition}[section]
\newtheorem{remark}{Remark}[section]
\newtheorem{definition}{Definition}[section]
\newtheorem{lemmad}{Lemma and definition}[section]

\numberwithin{equation}{section}

\def\Carre#1#2{\vbox{
   \hrule height .#2pt
   \hbox{\vrule width .#2pt height #1pt \kern #1pt
      \vrule width .#2pt}
   \hrule height .#2pt}}

\def\R{\mathbb{R}}

\def\H{\mathcal{H}}

\def\1{\raisebox{2pt}{\rm{$\chi$}}}

\def\1{\raisebox{2pt}{\rm{$\chi$}}}
\def\a{{\bf a}}
\def\b{{\bf b}}

\def\z{{\bf z}}

\def\R{\mathbb{R}}

\def\div{\mbox{div}\, }

\def\supp{\mbox{supp}\, }
\def\eps{\epsilon}

\def\res{\mathbin{\vrule height9pt width.1pt\vrule height.1pt width9pt}}

\renewcommand{\L}{\mathcal{L}}
\renewcommand{\H}{{\mathcal H}}

\begin{document}
\title[Qualitative behavior for flux-saturated equations]{Qualitative behavior for flux-saturated mechanisms: traveling waves, waiting time and smoothing effects}

\author{J. Calvo}

\address{Departamento de Tecnolog\' ia, Universitat Pompeu-Fabra. Barcelona, SPAIN}

\email{juan.calvo@upf.edu}

\thanks{J. Calvo, O. S\'anchez and J. Soler were supported in part by MICINN (Spain), project
 MTM2011-23384, and Junta de Andaluc\'{\i}a Project  FQM 954. J. Calvo is also partially supported by a Juan de la Cierva grant of the Spanish MEC.}

\author{J. Campos}
\address{Departamento de Matem\'atica Aplicada,
Facultad de Ciencias, Universidad de Granada. 18071 Granada, SPAIN}
\email{campos@ugr.es}
\thanks{J. Campos was supported in part by MICINN (Spain), project MTM2011-23652.}

\author{V. Caselles}
\address{Departamento de Tecnolog\' ia, Universitat Pompeu-Fabra. Barcelona, SPAIN}
\email{vicent.caselles@upf.edu}
\thanks{V. Caselles  was supported in part by MICINN (Spain), project MTM2009-08171, and
 also acknowledges the partial
support by GRC reference 2009 SGR 773, and by ''ICREA Acad\`emia'' prize for excellence in research funded both
by the Generalitat de Catalunya.}

\author{O. S\'anchez}
\address{Departamento de Matem\'atica Aplicada,
Facultad de Ciencias, Universidad de Granada. 18071 Granada, SPAIN}
\email{ossanche@ugr.es}

\author{J. Soler}
\address{Departamento de Matem\'atica Aplicada,
Facultad de Ciencias, Universidad de Granada. 18071 Granada, SPAIN}
\email{jsoler@ugr.es}


\dedicatory{To Vicent Caselles, in memoriam}


\maketitle

\begin{abstract} This paper is devoted to the analysis of qualitative properties of flux-saturated type operators in dimension one. Specifically, we study regularity properties and smoothing effects, discontinuous interfaces, the existence of traveling wave profiles, sub- and super-solutions and waiting time features. 
The aim of the paper is to better understand these kind of phenomena throughout two prototypic operators: The relativistic heat equation and the porous media flux-limited equation. As an important consequence of our results we deduce that solutions to the one-dimensional relativistic heat equation become smooth inside their support on the long time run. 

\end{abstract}


 \section{Introduction, preliminary and main results of the paper}
The aim of this paper is to investigate some qualitative properties  for a couple of  models arising in flux-saturated process. First we have
\begin{equation}
\label{parab}
\frac{\partial u}{\partial t}= \nu \, \left( \frac{|u|  (u^m)_x}{\sqrt{1+\frac{\nu^2}{c^2}|(u^m)_x|^2}} \right)_x,\quad  \nu, c>0,\, m\ge 1,
\end{equation}
which combines flux-saturation effects together with those of porous media flow. We will refer to this equation as Flux-Limited Porous Medium  equation (FLPME) as so was introduced in \cite{CCC} (see also \cite{CKR}). Here $u^m$ inside $(u^m)_x$ is meant to stand for $|u|^m \mbox{sign}\, (u)$. 

The second equation we will be concerned about in this paper is the so-called relativistic heat equation (RHE) \cite{Rosenau2, Brenier1}:
\begin{equation}
\label{problema}
\frac{\partial u}{\partial t} = \nu \, \left(\frac{|u|  u_x}{\sqrt{u^2 + \frac{\nu^2}{c^2} | u_x|^2}} \right)_x, \quad \nu, c>0.
\end{equation}

Both systems have differences and similarities regarding their qualitative properties, which makes them prototypes to analyze, compare and better understand the dynamics of flux-saturated mechanisms. 

Specifically, this paper deals with different smoothing effects of these flux-satura\-ted mechanisms as well as  with finite time extinction of discontinuous interfaces of solutions to the FLPME (while this kind of interfaces are preserved along the evolution of the RHE). Another interesting aspects reported in this paper is a waiting time phenomena for the FLPME. Under some circumstances the support will not spread until a sharp interface is formed by means of a mass redistribution process taking place inside the support. Once this happens the support will grow at a rate that depends on the parameters of the system.
 Moreover, there is a family of traveling wave solutions to FLPME that can be used to get accurate information about the aforementioned features.

Several aspects  concerning the mathematical theory of flux--saturated mechanism were introduced in the pioneering works \cite{CKR,Ros-Non,Rosenau2}. The theory for the existence of entropy solutions associated 
to flux--saturated equations has been widely developed in the framework of Bounded Variation functions, see \cite{ACMEllipticFLDE, ACM2005, Andreu2008}. The fact that the propagation speed of discontinuous interfaces is generically given by $c$ has been remarked in \cite{ARMAsupp, CMSV}; the precise Rankine--Hugoniot characterization of traveling jump discontinuities can be found in \cite{leysalto, CCC}. The problem of regularity has been previously treated in \cite{Andreu2008,ACMSV,Carrillo}, while diverse aspects of the waiting time phenomenon are addressed in  \cite{ACMSV,Carrillo,Giacomelli}. Applications of these ideas to diverse contexts such as Physics, Astronomy or Biology can  be found for instance in \cite{CMSV, LP,Rosenau2,VGRaS}.

The idea to analyze the regularity of solutions is to  transfer the problem to an auxiliary dual problem as was previously done in \cite{Carrillo} (see also  references therein). This dual problem has some regularity properties that are typical of uniformly ellip\-tic operators of second order. We are able to extend some of the results in \cite{Carrillo}, taking advantage of the fact that jump discontinuities determine dynamic regions where the quantity of mass is preserved. This enables to apply local regularity arguments for each of these regions separately and ultimately to show that there is a global smoothing effect for \eqref{problema} on the long time run that dissolves all singularities of the solution but those at its interface. This program applies to \eqref{parab} only partially, as the use of the dual formulation breaks down when interfaces become continuous. In fact, as shown in Section \ref{TW}, jump discontinuities (and particularly discontinuous interfaces) disappear  in finite time for FLPME.

In order to motivate the study of these systems let us give a scheme of how the FLPME and RHE  can be deduced from optimal mass transportation arguments. Following \cite{agueh2003existence, CCC2}, we can define the associated
Wasserstein distance
 between two probability distributions $\rho_0$ and $\rho_1$ by
$$
W^h_k (\rho_0, \,\rho_1) := \inf \left\{
\int_{\R^N \times \R^N} k\left(\frac{x-y}{h}\right) d\gamma(x, y), \quad \Gamma(\rho_0, \, \rho_1) \right\},
$$
where $h > 0$ and $\Gamma(\rho_0, \, \rho_1)$ is the set of probability measures in $\R^N \times \R^N$ whose marginals are $\rho_0$ and $\rho_1$.
Let $F : [0, \infty) \to [0, \infty)$ be a convex function and let ${\mathcal P}(\R^N)$ be the set of probability density
functions $\rho: \R^N \to [0, \infty)$.  Starting from $\rho^h_0 := \rho_0 \in {\mathcal P}(\R^N)$, we solve iteratively
$$
\inf_{\rho \in {\mathcal P}(\R^N)}
hW^h_k (\rho^h_{n-1}, \rho) + \int_{\R^N} F(\rho(x)) dx.
$$
This is a gradient descent with respect to the Wasserstein distance. Define $\rho^h(t) := \rho^h_n$ for $t \in [nh, (n + 1)h)$.  $k$ stands for the cost function, that we choose here as
$$
k(z)= \left\{ 
\begin{array}{ll}
c^2   \left(1 - {\sqrt{1-\frac{|z|^2}{c^2}}} \right), &  \mbox{ if }\  |z| \leq c , \\ 
+\infty , &\mbox{ if }\  |z| > c.
\end{array}
\right.
$$
The solution of the above iterative process converges as $h \to 0^+$  to the solution of the nonlinear diffusion
equation
$$
\frac{\partial u}{\partial t} = \div (u \nabla k^*(\nabla F'(u))), \qquad k^*(\xi) = c^2 \left(  \sqrt{1+ \frac{|\xi|^2}{c^2}}-1\right)
$$
Choosing the so-called Tsallis entropy $F(r) = {r^{m+1}}/{m}$, $m > 0$, we can deduce
\begin{eqnarray}\label{Tsa}
 \frac{\partial u}{\partial t}=  \frac{m+1}{m} \displaystyle{ \div \left( \frac{  u
 \nabla u^m} {\sqrt{ 1+  \left(\frac{m+1}{mc}\right)^2\left|\nabla u^m \right|^2}}\right ),}
\end{eqnarray}
which is a model for flux-limited-porous-media system.
In order to identify the above equation with the FLPME, let us first assume that we are in dimension one and analyze the effect of a rescaling for the equation \eqref{parab}. We introduce $\tilde{u}(t,x):=K u(Tt,Lx)$. Let us assume that $K=L$ in order to have mass preservation. If $u$ solves \eqref{parab} with constants $c,\nu$, then $\tilde{u}$ solves \eqref{parab} with constants $\tilde{c},\tilde{\nu}$, provided that
$$
\frac{\nu T}{K^m L^2}= \tilde{\nu} \quad \mbox{and}\quad \frac{\nu^2}{c^2 L^2K^{2m}}=\frac{\tilde{\nu}^2}{\tilde{c}^2}.
$$ 
If $c = \tilde{c}$ and $\nu = \tilde{\nu}$ then $T=L^{-m}$. In our case $ \tilde{\nu} = \frac{m+1}{m}= \frac{T}{L}$, which fixes the rescaling in order to connect \eqref{Tsa} with \eqref{parab}.

The same idea leads to the relativistic heat equation (\ref{problema}), using Gibbs--Boltzmann entropy $F(r) = r \ln r-r$.

The paper is structured as follows. Concluding Section 8 is in fact an Appendix where the reader can find an explanation of the notion of entropy solutions, plus several results on Rankine--Hugoniot conditions for moving fronts, comparison principles and useful estimates that we will use in the rest of the text and (for the most part) are scattered among the literature. This Appendix could be read right after this introduction or as a complement of the whole text, according to reader's convenience. Section 2 introduces a family of 
dual problems that serve as a tool to analyze regularity properties; important differences between \eqref{parab} and \eqref{problema} will become clear at this point. In Section 3 we construct traveling wave solutions to the FLPME, that we use right away to prove that jump discontinuities vanish in finite time. This implies in particular that initially discontinuous interfaces will eventually become continuous, as opposed to the case of the RHE. Section 4 is devoted to construct sub- and super-solutions of the FLPME which, in particular, imply that waiting time phenomena for the support growth are present in many cases. Section 5 concerns the smoothing effects for the RHE with a single singularity inside the support of the solution.
This study is then used in Section 6 to discuss regularity issues in the case of a finite number of singularities. Finally, Section 7 establishes some regularity properties for the FLPME before interfaces become continuous.

\section{The dual problem for the inverse distribution function}
\label{MDF}
In this section we associate to equations (\ref{parab}) and (\ref{problema}) dual problems that will allow later on to study 
local-in-time regularity properties in the interior of the support for both systems. As we proceed we will compare both cases and realize that there are several fundamental differences between their qualitative behaviors. To proceed, we introduce here a change of variables that was previously used in \cite{Carrillo} to study regularity properties of solutions to \eqref{problema}. Let us consider $u(t)$ an entropy solution of the Cauchy problem for \eqref{problema} which is smooth inside its support, which we assume to be connected (later on we will relax these conditions). Define 
$$
(a(t),b(t)):= (\min  \supp u(t),\max  \supp u(t)).
$$
Provided that
$$
M:= \int_{\R} u(0)\, dx
$$
(note that the total mass is preserved during evolution) we introduce an auxiliary function $\varphi(t,\cdot) : (0,M) \rightarrow (a(t),b(t))$ defined by
\begin{equation}
\label{rule2}
\int_{a(t)}^{\varphi(t,\eta)}u(t,x)\, dx = \eta, \quad \eta \in (0,M).
\end{equation}
Note that $\varphi(t,\cdot)$ is a bijection as long as $u(t)\ge 0$ has only isolated zeros inside its support. We will use this fact freely when displaying some formulas regarding sets of points in $(a(t),b(t))$, which can be seen as images of sets in $(0,M)$. 

Now we let $v(t,\eta) := \frac{\partial \varphi}{\partial \eta}(t,\eta)$, which relates to $u(t,x)$ by means of
\begin{equation}
\label{rule1}
v(t,\eta) = \frac{1}{u(t,\varphi(t,\eta))}.
\end{equation}
Note also that
\begin{equation}
\label{regtransfer}
\frac{\partial v}{\partial \eta}(t,\eta)= - v^3(t,\eta) \frac{\partial u}{\partial x}(t,\varphi(t,\eta)).
\end{equation}
This function $v$ satisfies the following equation:
\begin{equation}
\label{RHEdv}
 v_{t} = \left(\frac{\nu v_\eta}{\sqrt{v^4 + \frac{\nu^2}{c^2}(v_\eta)^2}} \right)_\eta,\quad t>0,\ \eta \in (0,M).
\end{equation}
The boundary conditions at $\partial (0,M)$ depend on the behavior of $u(t)$ at the interface. An important case is that in which $u(t)$ is compactly supported and the slopes at the interfaces are $+\infty$ and $-\infty$ (according to Theorem \ref{th2} and Proposition \ref{p4}). Under these circumstances, if we denote by $n$ the outer unit normal to $(0,M)$, that is $n(0) = -1$ and $n(M)=1$,
the natural boundary conditions for (\ref{RHEdv}) are
\begin{equation}
\label{verticalBC0}
\frac{ \nu v_\eta}{\sqrt{v^4 + \frac{\nu^2}{c^2}(v_\eta)^2}} n = c \qquad \mbox{at }
\eta\in\partial(0,M)\,,
\end{equation}
with $\partial(0,M)=\{0,M\}$. 

The same rules to pass to the dual formulation apply for any equation of the form $u_t=[\a(u,u_\eta)]_\eta$. In the particular case of \eqref{parab}, we get
$$
\varphi_t = \frac{\nu m \varphi_{\eta\eta}}{\sqrt{(\varphi_\eta)^{4+2m}+ \frac{\nu^2 m^2}{c^2}(\varphi_{\eta\eta})^2}}
$$
and
\begin{equation}
\label{mm}
v_t = \left(\frac{\nu m v_\eta}{\sqrt{v^{4+ 2m} + \frac{\nu^2 m^2}{c^2}(v_\eta)^2}}\right)_\eta \quad t>0,\ \eta \in (0,M).
\end{equation}
The boundary conditions at $\partial (0,M)$ depend on the behavior of $u(t)$ at the interface.  If the 
solution exhibits jumps at the boundaries the natural ones are given by 
\begin{equation}
\label{verticalBC}
\frac{\nu m v_\eta}{\sqrt{v^{4+2m} + \frac{\nu^2 m^2}{c^2}(v_\eta)^2}} n = c \qquad \mbox{at }
\eta\in\partial(0,M)\,,
\end{equation}
as explained above. 

We may now normalize the solutions to  \eqref{RHEdv}--\eqref{verticalBC0} and  \eqref{mm}--\eqref{verticalBC}. Let 
\begin{equation}
\label{0scala}
\bar {v}(t,\eta):= v(\nu t , \frac{\nu}{c}\eta) 
\end{equation}
for the case of \eqref{RHEdv}--\eqref{verticalBC0} and 
\begin{equation}
\label{mscala}
\bar {v}(t,\eta):= v(\nu m t , \frac{\nu m}{c}\eta) 
\end{equation}
for that of \eqref{mm}--\eqref{verticalBC}.
Then, irrespective of the case, $\bar{v}$ verifies the following general dual formulation:
\begin{equation}
\label{common}
\bar{v}_t = \left(\frac{ \bar{v}_\eta}{\sqrt{\bar{v}^{4+ 2m} + (\bar{v}_\eta)^2}}\right)_\eta \quad t>0,\ \eta \in (0,M)\, ,
\end{equation}
with boundary conditions
\begin{equation}
\label{commonBC}
\frac{\bar{v}_\eta}{\sqrt{\bar{v}^{4+2m} + (\bar{v}_\eta)^2}} n = 1 \qquad \mbox{at }
\eta\in\partial(0,M)\,,
\end{equation}
where $m \geq 1$ for FLPM and  $m=0$ for RHE. We maintain the notation $M$ for the rescaled mass and we will work with the rescaled systems from now on.
  
In general, solutions of \eqref{common}--\eqref{commonBC} do not fulfill the boundary conditions in the classical sense. The notion of weak trace as introduced in \cite{Dirichlet} (see Definition 4 there in particular) should be used to give a meaning to \eqref{commonBC}, which is the meaning that should be attached to \eqref{commonBC} --and also to \eqref{verticalBC0}, \eqref{verticalBC}-- during Section \ref{MDF}. We will refrain to do so here though, since we won't require this weak form of the boundary conditions for future sections. In fact, we will be able to show that  boundary conditions \eqref{commonBC} can be given a more tractable formulation as traces of functions of bounded variation under some particular circumstances ({see Lemma \ref{milestone} and Corollary \ref{dual_milestone} below}); for practical purposes this will be enough, as we will be always working in this easier setting.
That being said, the first step in our analysis is to prove a regularity result for \eqref{common}--\eqref{commonBC}:
\begin{theorem}
\label{thm:regv1D}
Let $m \geq 0$.
 Assume that $\bar{v}_0\in W^{1,\infty}(0,M)$, $\bar{v}_0\geq \alpha_1 > 0$.
Then there exists some $0<T^*<\infty$ (depending on $\bar{v}_0$) and a smooth solution $\bar{v}$ of \eqref{common} in
$(0,T^*)\times (0,M)$ with $\bar{v}(0,\eta)= \bar{v}_0(\eta)$ and satisfying the boundary conditions \eqref{commonBC}. 
\end{theorem}
\begin{proof}
To prove this claim, we consider the following approximated Cauchy
problem
\begin{align}
&v_t =  \, \left( \frac{v_\eta}{ \sqrt{ v^{4+2m} + (v_\eta)^2}} \right)_\eta
+\epsilon v_{\eta \eta} \qquad  t\in (0,T),\, \, \eta\in (0,M)\,,\label{RHEdvA} \\
&\left(\frac{v_\eta}{\sqrt{v^{4+2m} + (v_\eta)^2}} +\epsilon v_\eta\right) n =
1 -\epsilon^{1/3}, \qquad  t\in (0,T),\, \, \eta\in
\partial (0,M),\label{RHEdvAB}
\end{align}
where $\epsilon >0$, for any $T>0$. For simplicity, of notation we will use $v$ instead of $\bar{v}$ along the proof. We proceed in several steps. We start proving some formal estimates in
Steps 1 and 2  that are used later 
to state the existence of solutions of
(\ref{RHEdvA})--\eqref{RHEdvAB} in Step 3. It is very important to remark at this point that the estimates are local in time. Set
$$
\a(z,\xi) = \frac{\xi}{\sqrt{  z^{4+2m} +  (\xi)^2}},\qquad z \geq 0,
\,\xi \in\R.
$$
Let us observe that
\begin{equation}\label{fi1}
\a(z,\xi)  \xi \geq  |\xi|- z^{2+m}.
\end{equation}

\noindent {\em Step 1. $L^\infty$ bounds on $v$  from above and below independent of $\epsilon$.}
\label{eins}
Let us construct a super-solution to the Cauchy problem \eqref{RHEdvA}--\eqref{RHEdvAB} having the following form:
$$
V(t,\eta) = B(t) -\sqrt{\epsilon^\frac23 + \eta (M-\eta)}.
$$
Here $B(\cdot)$ is an increasing function of time to be determined. Since $v_0$ is bounded above we can choose $B(0)> 1$ and  such that $V(0,\eta) \ge B(0) - \sqrt{\epsilon^\frac23  + \frac{M}{4}}  \ge v_0(\eta)$.

Direct computations give 
$$
V_\eta = \frac{\eta - \frac{M}2}{\sqrt{ \epsilon^\frac23 + \eta (M-\eta) }}  \, , \quad
V_{\eta \eta} = \frac{\epsilon^\frac23 + \frac{M}{4}}{ \left(\epsilon^\frac23 + \eta (M-\eta)\right)^\frac32 } \, , \quad
a(V,V_\eta) 
= \frac{\eta - \frac{M}2}{D(t,\eta)}\, , \quad
$$
and
$$
\left(a(V,V_\eta)\right)_\eta 
= \frac{1}{D} + \frac{(\eta - \frac{M}2)^2}{D^3}\left(V^{4+2m} - 1 -(2+m) V^{3+2m} \sqrt{ \epsilon^\frac23 + \eta (M-\eta) }\right)\, . \quad
$$
Here $D =D(t,\eta)=  \sqrt{V^{4+2m} (\epsilon^\frac23 + \eta (M-\eta)) + (\eta - \frac{M}2)^2}$ is bounded from 
below, since 
$$
D(t,\eta) \geq D(0,\eta) > \sqrt{\alpha_1^{4+2m} (\eta (M-\eta)) + (\eta - \frac{M}2)^2 }>  min\{1,\alpha_1^{2+m}\} \frac{M}2. 
$$ 
Then it can be shown that
\begin{equation}
\label{double_trouble}
|\left(a(V,V_\eta)\right)_\eta  | \leq C_2+ C_1 B(t)^{4+2m}, \quad |\left(a(V,V_\eta)\right)_\eta  | \leq C_3(\epsilon),
\end{equation}
where $C_1$ and $C_2$ are positive constants independents of  $\epsilon$, $C_3$ blows up when $\epsilon \to 0$ and $\eta \in (0,M)$.  
On the other hand
$$|\epsilon V_{\eta \eta} | \leq  \epsilon  \frac{\epsilon^\frac23 + \frac{M}{4}}{\epsilon } \leq C_4 $$ 
for bounded values of $\epsilon$. Thus, if we use the second estimate in \eqref{double_trouble} we get easily a global-in-time super-solution. This provides a global $L^\infty$-bound that is not uniform in $\epsilon$. If we want to get a uniform bound we must use the first estimate in \eqref{double_trouble}. The price to pay is that we can only construct a local-in-time super-solution: According to that estimate, there must hold that 
$$\left(a(V,V_\eta)\right)_\eta + \epsilon V_{\eta \eta}  \leq C_2+ C_4+ C_1 B(t)^{4+2m} \leq B'(t) = V_t \, .$$
Such a function $B(t)$ exists only in a finite time interval $(0,T^*)$ for a certain $T^*< \infty$ (depending on $m$,  $B(0)$, $C_1$ and $C_2$).
In order to conclude that the function $V$ determined in this way is a super-solution we have to check that
$$\left(a(V,V_\eta + \epsilon V_{\eta}\right)  n \geq 1 -\epsilon^\frac13 $$
for $t \in (0,T^*)$. This is easily seen for $\epsilon$ small enough as was done in \cite{Carrillo}.

It is easily shown that the constant function $\bar{V} = \alpha_1$ is a sub-solution. 
Thanks to the classical weak comparison principle we have that  any  solution $v$ to the Cauchy problem \eqref{RHEdvA}--\eqref{RHEdvAB} is bounded from below by $\alpha_1$ and from above by $V(t,\eta)$ for $t$ smaller than $T^*$.

\noindent {\em Step 2. $L^p$ bounds on $v_\eta$ independent of $\epsilon$.}
\label{zwei}
By integrating \eqref{RHEdvA} and using the boundary conditions \eqref{RHEdvAB} we can deduce that 
$$\int_0^M v(t,\eta) d \eta = \int_0^M v(0,\eta) d \eta + 2(1-\epsilon^\frac13) t. $$
This can be combined with the bounds given in Step 1. to assure  that  $v(t,\cdot) \in L^p[0,M]$ for any $p \in[1,\infty]$ and 
$t \in [0,T^*)$  by interpolation.

For any $p \in [1,\infty)$ and $t \in [0,T^*)$ we have 
\begin{eqnarray}
&&\int_0^t \int_0^M a(v,v_\eta) (v^p)_\eta \,  d \eta\,  dt  + \epsilon p  \int_0^t \int_0^M v^{p-1} (v_\eta)^2 \, d \eta \,   dt
 \nonumber \\ 
&& \hspace{0.5cm}=  \frac{1}{p+1} \left(\int_0^M v^{p+1}(0,\eta) \, d \eta -\int_0^M v^{p+1}(t,\eta) \, d \eta\right) 
 \nonumber \\ 
&& \hspace{0.8cm}+ (1-\epsilon^\frac13) \int_0^t \big( v^p(s, 0) + v^p(s, M) \big) \, ds \nonumber.
\end{eqnarray}
From  (\ref{fi1}) we get
$$\int_0^t \int_0^M a(v,v_\eta) (v^p)_\eta \,  d \eta\,  dt 
\geq 
\int_0^t \int_0^M |(v^{p})_{\eta}| \, d \eta \,   dt - \int_0^t \int_0^M v^{p+1+m} \, d \eta \,   dt ,$$
which allow us to conclude that 
$$
\int_0^t \int_0^M |(v^{p})_{\eta}| \, d \eta \,   dt \leq C(t,p)\, ,
$$
for any $p \in [1,\infty)$ and $t \in [0,T^*)$.

\noindent {\em Step 3. Existence of smooth solutions for the Cauchy problem \eqref{RHEdvA}-\eqref{RHEdvAB}.}
\label{drei}
The existence of solutions of (\ref{RHEdvA})-\eqref{RHEdvAB}
follows from classical results, for instance those in \cite{Ladyparabolico} and \cite[Theorem
13.24]{lieberman2005second}. Note that thanks to the a priori bounds stated previously, the flux 
$$
a_{\alpha_1}(v,v_\eta):= \frac{v_\eta}{\sqrt{\sup \{\alpha_1,|v|^{4+2m}\}+(v_\eta)^2}}
$$
can be used. Note that $v_0$ may not satisfy \eqref{verticalBC}. Details on how to proceed to amend this are provided in \cite{Carrillo}.

Let $v_\epsilon$ be the solution of the Cauchy problem
(\ref{RHEdvA})--\eqref{RHEdvAB}. Then the first
derivatives of $v_\epsilon$ are H\"older continuous up to the boundary. Moreover, for
$g=v_{\epsilon xx},v_{\epsilon t}$, we have
$$
\sup_{x\neq y}
\left\{\min(d((x,t),\mathcal{P}),d((y,s),\mathcal{P}))^{1-\delta}
\frac{|g(x)-g(y)|}{(|x-y|^2+|s-t|)^{\alpha/2}}\right\}<\infty ,
$$
for some $\alpha,\delta > 0$. Here $\mathcal{P}$ is the parabolic
boundary of $(0,M)\times (0,T)$, that is $[0,M] \times \{0\} \cup
\{0,M\}\times (0,T)$, and $d(\cdot,\mathcal{P})$ denotes the
distance to $\mathcal{P}$. The non-uniform global bounds derived in Step 2 were used here. On the other hand, by the interior
regularity result \cite[Chapter V, Theorem 3.1]{Ladyparabolico}, the
solution is infinitely smooth in the interior of the domain. Here the smoothness bounds depend on $\epsilon$.

\noindent {\em Step 4. A local Lipschitz bound on $v_\epsilon$ uniform on $\epsilon$.}
\label{vier}
For simplicity of notation, let us write $v$ instead of $v_\epsilon$.
Let $w= |v_\eta|^2\phi^2$ where $\phi(\eta) \geq 0$ is smooth with compact support.

The estimates we are interested in are direct consequence of the inequality
\begin{equation}\label{wtestimates}
w_t \leq A(t,\eta) w_{\eta \eta} + B(t,\eta) w_\eta + C(t,\eta) w + f(t,\eta) ,
\end{equation}
where $A$, $B$, $C$, $f$ will be determined in the sequel.

Differentiating $|v_\eta|^2$ and multiplying by $\phi^2$ we get 
\begin{eqnarray*}
\frac12 w_t &= &
a_{zz} v_\eta^3 \phi^2 
+ 2 a_{z \xi } v_\eta^2 v_{\eta \eta} \phi^2 
+ a_{z } v_\eta v_{\eta \eta} \phi^2
\\ &&+ a_{\xi \xi } v_\eta v_{\eta \eta}^2 \phi^2
+ a_{\xi} v_\eta v_{\eta \eta \eta} \phi^2
+ \epsilon  v_\eta v_{\eta \eta \eta} \phi^2 ,
\end{eqnarray*}
where
\begin{eqnarray*}
&a_z = - \frac{(2+m) z^{3+2m} \xi}{(z^{4+2m} + \xi^2)^\frac32} \, , 
& a_{zz} =  -\frac{(2+m) (3+2m) z^{2+2m} \xi}{(z^{4+2m} + \xi^2)^\frac32} +\frac{3 (2+m)^2  z^{6+4m} \xi}{(z^{4+2m} + \xi^2)^\frac52} \, , \quad
\\
& a_\xi =  \frac{z^{4+2m} }{(z^{4+2m} + \xi^2)^\frac32} \, ,
& a_{z \xi} =  -\frac{(2+m)  z^{3+2m}}{(z^{4+2m} + \xi^2)^\frac32} +\frac{3 (2+m)  z^{3+2m} \xi^2}{(z^{4+2m} + \xi^2)^\frac52} \, , \quad
\end{eqnarray*}
and
$$a_{\xi \xi} =  - \frac{3 z^{4+2m} \xi}{(z^{4+2m} + \xi^2)^\frac52}\,  . $$
Then,
\begin{eqnarray*}
a_{zz} \phi^2 v_\eta^3 
&=&   -\frac{(2+m) (3+2m) v^{2+2m} v_\eta^2 w }{(v^{4+2m} + v_\eta^2)^\frac32} +\frac{3 (2+m)^2  v^{6+4m} v_\eta^2 w }{(v^{4+2m} + v_\eta^2)^\frac52} 
\\ &\leq& 3 (2+m)^2 v^m w
\, ,
\\
 2 a_{z \xi } v_\eta^2 v_{\eta \eta} \phi^2 &=&
 a_{z \xi} v_\eta w_\eta - 2 a_{\xi  z } \phi \phi_\eta v_\eta^3
 \leq
  a_{z \xi} v_\eta w_\eta
  + 6 (2+m)  v^{3+2m} \phi |\phi_\eta| \, ,
\end{eqnarray*}
and  computing $w_{\eta}, \ w_{\eta \eta}$ we have
\begin{eqnarray*}
 a_{z } v_\eta v_{\eta \eta} \phi^2&=&
 \frac{1}{2} a_z w_\eta  - a_z v_{\eta}^2\phi \phi_\eta
 \le \frac{1}{2} a_z w_\eta + (2+m)v^{3+2m}\phi |\phi_\eta| \, ,
\\
a_{\xi \xi } v_\eta v_{\eta \eta}^2 \phi^2 &=& \frac{w_\eta}{2} a_{\xi \xi}v_{\eta \eta} - a_{\xi \xi} v_{\eta \eta} \phi \phi_\eta (v_\eta)^2\, ,
\\
v_\eta v_{\eta \eta \eta}  \phi^2 &=& \frac{w_{\eta \eta}}{2} - (\phi_\eta^2 + \phi \phi_{\eta \eta}) v_\eta^2- v_{\eta \eta}^2 \phi^2 -4 v_\eta v_{\eta \eta} \phi \phi_\eta\, .
\end{eqnarray*}
From these identities, and using that 
$$
|a_{\xi \xi} v_{\eta \eta} \phi \phi_\eta (v_\eta)^2|+
|4 a_{\xi} v_{\eta \eta} \phi v_\eta \phi_\eta|
\leq 
\frac12 {a_\xi} v_{\eta \eta}^2 \phi^2 + \frac{a_{\xi \xi}^2}{2a_\xi} \phi_\eta^2 (v_\eta)^4
+ \frac12 {a_\xi} v_{\eta \eta}^2 \phi^2 + 8 a_\xi v_\eta^2 \phi_\eta^2
$$
we have:
\begin{eqnarray*} 
&& \hspace{-1cm} a_{\xi \xi } v_\eta v_{\eta \eta}^2 \phi^2
+ a_{\xi} v_\eta v_{\eta \eta \eta} \phi^2
\\ &&\leq 
\frac{a_{\xi \xi}v_{\eta \eta}}{2}  w_\eta
+
\frac92 v^{2+m}  \phi_\eta^2
+
  a_{\xi} \frac{w_{\eta \eta}}{2} +v^{2+m} (\phi_\eta^2 + \phi |\phi_{\eta \eta}|)
  + 
 8 v^{2+m} \phi_\eta^2 .
\end{eqnarray*}
Finally, we estimate 
\begin{eqnarray} 
\epsilon  v_\eta v_{\eta \eta \eta} \phi^2
&\leq& 
 \epsilon \left(
 \frac{w_{\eta \eta}}{2} - \phi_\eta^2v_\eta^2  + \frac12 \phi^2v_\eta^2  +\frac12  \phi_{\eta \eta}^2 v_\eta^2
 - v_{\eta \eta}^2 \phi^2 +4 \phi_\eta^2  v_\eta^2    + v_{\eta \eta}^2 \phi^2 
 \right)  \nonumber \\
 & = &  
 \frac{\epsilon}{2} w_{\eta \eta} +  \frac{\epsilon}{2}w +\frac{\epsilon}{2}  \phi_{\eta \eta}^2 v_{\eta}^2 + 3\epsilon v_{\eta}^2 \phi_{\eta}^2 . \nonumber 
\end{eqnarray}
Putting together all  these estimates we get (\ref{wtestimates}) with 
\begin{eqnarray*}
A(t,\eta)&=&
 \frac12 (a_{\xi}+ \epsilon) \, ,
\\
B(t,\eta) &=& a_{z \xi} v_\eta +\frac12 a_z +\frac12 a_{\xi \xi} v_{\eta \eta}\, ,
\\
C(t,\eta) &=& 3 (2+m)^2 v^m + \frac12 \epsilon\, ,
\\
f(t,\eta) &=& 7 (2+m)  v^{3+2m} \phi |\phi_\eta| +v^{2+m} (\frac{27}{2}\phi_\eta^2 + \phi |\phi_{\eta \eta}|) +
\epsilon v_{\eta}^2 ( \frac{1}{2}  \phi_{\eta \eta}^2  + 3 \phi_{\eta}^2) \, . 
\end{eqnarray*}

We show next that $C(t,\eta)$ and $f(t,\eta)$ are uniformly bounded in $(0,T^*)$; this ensures $L^\infty$ bounds on $w$ in $(0,T^*)$ which are independent of $\epsilon$, thanks to the maximum principle. In fact, recall that $T^*$ was introduced in Step 1. The super-solution introduced there provides us with the required bound for $C(t,\eta)$. Now we derive a bound for $f(t,\eta)$. Taking $\phi =1$ in (\ref{wtestimates}) yields
$$w_t \leq  \frac12 (a_{\xi}+ \epsilon) w_{\eta \eta} + \left(a_{z \xi} v_\eta +\frac12 a_z +\frac12 a_{\xi \xi} v_{\eta \eta}\right) w_\eta + \left(3 (2+m)^2 v^m + \frac12 \epsilon \right) w \, .$$
On the other hand, taking into account the boundary conditions (\ref{RHEdvAB}) and estimate (\ref{fi1}), there follows that 
$$
|v_{\eta}| -v^{2+m}  +\epsilon v_{\eta}^2  \leq a(v,v_\eta)v_\eta +\epsilon v_{\eta}^2 \leq (1-\epsilon^\frac13)  |v_{\eta}|.
$$
Hence $ \epsilon v_{\eta}(t,\cdot) ^2 \leq v^{2+m} $ and consequently  $$\epsilon w(t,\cdot) \leq C ,\quad \mbox{on } (0,T^*) \times \partial (0,M)\, .$$ Then, the maximum principle shows that 
\begin{equation}
\label{viscbound}
\epsilon \|v_{\eta}(t, \cdot)\|_{L^{\infty}(0,M)}^2 \leq \tilde C , \quad \mbox{ in  } (0,T^*) \times (0,M)
\end{equation}
where  $\tilde C$ is a constant independent of $\epsilon$. From this, we  can deduced uniform bounds for $f(t,\eta)$ independent
of $\epsilon$ in $(0,T)\times (0,M)$,  for any $T< T^*$. Thus, we have proved local Lipschitz bounds on $v_\eta$ which are uniform in $\epsilon$ and hold for $(0,T^*)$.

\medskip\noindent {\em Step 5. Conclussion.}
\label{funf}
Thanks to the smoothness results stated in Step 3 and the local
uniform bounds on the gradient in Step 4, the classical interior
regularity results in \cite[Chapter V, Theorem 3.1]{Ladyparabolico} shows
uniform (in $\epsilon$) interior bounds for any space and time
derivative of $v_\epsilon$. These
regularity bounds allow to pass to the limit and
obtain a solution $v$ of
\begin{equation*}
v_t =  \, \left( \frac{ v_x}{\sqrt{v^{4+2m} + (v_x)^2}} \right)_x
\qquad  \hbox{\rm in $\mathcal{D}^\prime((0,T^*)\times (0,M))$}\,
\end{equation*}
(plus boundary conditions). This is done in the same fashion as in \emph{Step 7} in the proof of Theorem 2.1 in \cite{Carrillo}. The only important difference is that uniform bounds on 
$$
 \a_\epsilon =
\frac{ v_{\epsilon x}}{\sqrt{v_{\epsilon}^{4+2m} + (v_{\epsilon x})^2}}
+ \epsilon v_{\epsilon x}
$$
independent of $\epsilon$ are obtained just in $(0,T^*)\times (0,M)$, instead of 
$(0,T)\times (0,M)$ for any $T>0$ (see \eqref{viscbound}). 
Boundary conditions \eqref{commonBC} in weak form are recovered using the convergence result given by Lemma 10 in \cite{Dirichlet}. 
\end{proof}

The relevance of this result lies in the fact that it allows to construct an entropy solution for either \eqref{parab} or \eqref{problema} that enjoys certain nice properties. To see how, let $u_0\in  L^\infty(\R)$ with $u_0(x) \geq \kappa > 0$ for $x\in [a,b]$, and $u_0(x)=0$ for $x\not\in [a,b]$. Assume that $u_0\in W^{1,\infty}([a,b])$. Let $v_0(\eta)$ be defined in  $(0,M)$ according to \eqref{rule1}.
Then we let $u(t,x)$ be defined in $[a-ct,b+ct]$ by \eqref{rule2}--\eqref{rule1} and \eqref{mscala} or \eqref{0scala} depending on the case, while we set $u(t,x) = 0$, $x\not \in [a-t,b+t]$, $t\in (0,T^*)$. Notice that $u(t,x)\geq \kappa(t) > 0$ for any $x\in (a-t,b+t)$ and any $t <T^*$. Under that circumstances, a straightforward adaptation of Proposition 2.5 in \cite{Carrillo} yields the following result. 
\begin{proposition}
\label{prop:suficient}
Let $m>1$ or $m=0$ and let $\bar{v}$ be a solution given
by Theorem {\rm\ref{thm:regv1D}}. Now consider $u$ to be defined by \eqref{rule1}--\eqref{mscala} if $m>1$ or \eqref{rule1}--\eqref{0scala} if $m=0$. Then $u\in C([0,T^*),L^1(\R))$, so that
$u(0)=u_0$ and satisfies
\begin{itemize}
\item[(i)] $u(t)\in BV(\R)$, $u(t)\in W^{1,1}(a-ct,b+ct)$ for almost
any $t\in (0,T^*)$, and $u(t)$ is smooth inside its support,

\item[(ii)] $u_t = \z_x$ in $\mathcal{D}^\prime((0,T^*)\times \R)$,
where 
\begin{eqnarray*}
\z(t) &= \displaystyle \frac{\nu u(t) (u^m)_x(t)}{\sqrt{1+ \frac{\nu^2}{c^2}((u^m)_x(t))^2}} 
\\ \z(t) &= \displaystyle \frac{\nu u(t) u_x(t)}{\sqrt{u(t)^2+ \frac{\nu^2}{c^2}((u_x(t))^2}} 
& \quad (\mbox{case}\ m=0).
\end{eqnarray*}
\item[(iii)] $u(t,x)$ is the entropy solution of \eqref{parab} (resp. \eqref{problema}) with
initial data $u_0$ in $(0,T^*)$.
\item[(iv)] $u(t)$ is strictly positive inside its support. 
\end{itemize}
\end{proposition}

\subsection{Global statements for the relativistic heat equation}
\label{rhestatements}
Recall that Theorem \ref{thm:regv1D} holds only in a finite time interval $(0,T)$, due to the fact that we were not able to obtain global-in-time uniform bounds on $v$. This cannot be helped in the case of \eqref{mm}--\eqref{verticalBC}, because if such a global bound is to exist, then $u$ would be strictly positive in its support for every time instant and this  would contradict forthcoming Corollary \ref{finitecontact}.
On the contrary, we know that solutions of the relativistic heat equation which are initially strictly positive everywhere in their support remain so during evolution. Thus, switching back to \eqref{RHEdv}, we should be able to prove a global uniform bound on the associated solutions. 
\begin{proposition}
\label{pr22}
Assume that $v_0\in W^{1,\infty}(0,M)$, $v_0\geq \alpha_1 > 0$. Given any $T>0$, there exists a smooth solution $v$ of \eqref{RHEdv} in
$(0,T)\times (0,M)$ with $v(0,\eta)= v_0(\eta)$ and satisfying 
boundary conditions \eqref{verticalBC0}.
\end{proposition}
\begin{proof}
Apply Theorem \ref{thm:regv1D} to $v_0$, obtaining a smooth solution $v^1$ defined on some time interval $(0,T^1)$. Then, we use Proposition \ref{prop:suficient} to deduce the estimate 
$$
\|v^1(t)\|_\infty \le 1/(\inf_{\supp u(t)} u(t))\quad \mbox{for any}\ t<T^1.
$$ 
 Thanks to Proposition \ref{subsol} we obtain that 
$$
\inf_{\supp u(t)} u(t) \ge e^{-\beta_1 t -\beta_2 t^2} \inf_{\supp u_0} u_0/2
$$
for some constants $\beta_1,\beta_2 >0$ (to be precise these constants get larger as $|\supp u_0|$ does, but given that $T<\infty$ has been fixed, the measure of $\supp u(t)$ is controlled for any $t<T$ and we can neglect this dependency in the sequel).
Hence $\sup_{t \in (0,T^1)} \|v^1(t)\|_\infty <\infty$ and $v^1$ can be extended smoothly to a solution of \eqref{RHEdv} in $[0,T^1]$. We let now
$$
I_1:= e^{-\beta_1 T^1 -\beta_2 (T^1)^2} \inf_{\supp u_0} u_0/2.
$$
This allows to use again Theorem \ref{thm:regv1D} with $v^1(T^1)$ as initial condition, obtaining a a new solution $v^2$ defined on some interval $[T^1,T^1+T^2)$, with $T^2$ depending only on $I_1$. As before,
$$
\inf_{\supp u(t)} u(T^1+t) \ge I_1 e^{-\beta_1 t -\beta_2 t^2}/2 
$$
for any $t \in (0,T^2)$. Then we can extend $v^2$ to $T^1+T^2$ with finite uniform bounds. Proceeding as before, we set 
$$
I_2:= e^{-\beta_1 T^2 -\beta_2 (T^2)^2} I_1/2.
$$
We may repeat this at will. To prove our statement we must show that $T^1+T^2+\cdots$ diverges. Arguing by contradiction, let $T^*=\sum_{i=1}^\infty T^i$. Superposing the various solutions $v^i$ we define a solution $v$ in $(0,T^*)$. Using Proposition \ref{prop:suficient} we obtain a solution of defined in $(0,T^*)$. Resorting again to Proposition \ref{subsol},
$$
\inf_{\supp u(t)} u(t) \ge e^{-\beta_1 t -\beta_2 t^2} \inf_{\supp u_0} u_0/2
$$
for any $t \in (0,T^*)$ and thus 
$$
\sup_{t \in (0,T^*)} \|v(t)\|_\infty <\infty.
$$
Then $v$ can be extended smoothly to $[0,T^*]$. This allows to use Theorem \ref{thm:regv1D} one more time and extend the definition of $v$ beyond $T^*$, obtaining the desired contradiction.
\end{proof}
\begin{corollary}
Let $\bar{v}$ be a global-in-time solution given
by Theorem {\rm\ref{thm:regv1D}} and Proposition \ref{prop:suficient} for the case $m=0$. Now consider $u$ to be defined by \eqref{rule1}--\eqref{0scala}. Then $u\in C([0,T],L^1(\R))$ for any $T<\infty$, hence
$u(0)=u_0$. Moreover, statements \emph{(i)}, \emph{(ii)}, \emph{(iii)} and \emph{(iv)} in Proposition \ref{prop:suficient} hold with $(0,T^*)$ replaced by $(0,T)$.
\end{corollary}

In order to perform our regularity analysis in Section \ref{CasoModelo} below, we need to sharpen the statement of Theorem \ref{thm:regv1D}. First we need a definition.
\begin{definition}
Let $v$ be the (weak) solution of \eqref{RHEdv} with suitable boundary conditions. Given $0\le t<T$, we say that $x\in (0,M)$ is a singular point for $v(t)$ if $v(t,\cdot)$ is not Lipschitz continuous at $x$. We write $S_v(t)$ for the set of singular points of $v(t)$. 
\end{definition}
Hereafter we will use $m$ as the ``spatial'' variable for \eqref{RHEdv}, in order to stress that we are dealing only with \eqref{problema} this time.
Our improvement on Proposition \ref{pr22} goes as follows:
\begin{theorem}
\label{extendedMDF}
Assume that
$v_0\in BV(0,M)$, $v_0\geq \alpha_1 > 0$. Assume also that $S_v(0)$ is finite and $v_0\in W_{loc}^{1,\infty}((0,M)\backslash S_v(0))$. Then, for any $T>0$ there exists a weak solution of \eqref{RHEdv}--\eqref{verticalBC0}  in
$(0,T)\times (0,M)$ with 
$v(0,x)= v_0(x)$. Moreover:
\begin{enumerate}

\item  $S_v(t_2)\subset S_v(t_1)$ for any $t_2 >t_1\ge 0$. Thus, 
{ $v(t)\in W_{loc}^{1,\infty}((0,M)\backslash S_v(0))$ for every $0<t<T$.}
\label{nmv}

\item {$v(t)$ is smooth in $(0,M)\backslash S_v(t)$ for every $0<t<T$ (in fact $v$ is smooth in $\cup_{0<t<T}(\{t\}\times ((0,M)\backslash S_v(t)))$).}

\item $v(t) \in BV(0,M)$ a.e. $0<t<T$.
\label{w11}

\end{enumerate}
\end{theorem}
\begin{proof} 
In order to show this result we approximate the initial datum by Lipschitz functions, to which we apply Proposition \ref{pr22} --modulo \eqref{0scala}. Let $\{v_{0,\eps}\}\subset W^{1,\infty}(0,M)$ be a sequence of functions 
verifying \eqref{verticalBC0} such that $v_{0,\eps}\ge \alpha_1$ and $v_{0,\eps} \to v_0$ in 
\newline
$W^{1,\infty}_{loc}((0,M)\backslash S_v(0))\cap BV(0,M)$ when $\eps \to 0$. Then Proposition \ref{pr22} ensures that for each $\eps>0$ there exists a smooth solution $v_\eps$ of \eqref{RHEdv} in
$(0,T)\times (0,M)$ with $v_\eps(0,m)= v_{0,\eps}(m)$ and satisfying 
boundary conditions \eqref{verticalBC}. As $\eps \to 0$ the derivatives of $v_{0,\eps}$ will blow up in the vicinity of $S_v(0)$, but keep in mind that $v_{0,\eps}$ is locally Lipschitz inside $(0,M)\backslash S_v(0)$ with bounds independent of $\eps$. In the following we skip the sub-index $\eps$ except at some places in which we find useful to keep it.

\noindent {\em Step 1: Integral bounds.}
To begin with, using the comparison principles in the proof of Theorem \ref{thm:regv1D}, \emph{Step 1} together with Proposition \ref{pr22}, we deduce that 
$$
\alpha_1 \le v(t,m) \le C,\quad (t,m) \in [0,T]\times [0,M],
$$
being $C>0$ some positive constant depending only on $u_0$ and $T$. Next, it is easily seen that 
$$
\int_0^M v(t,m)\ dm = \int_0^M v_0(m)\ dm +2 c t.
$$
Thus, $v_\eps\in L^\infty(0,T,L^p(0,M))$ for any $1\le p\le \infty$ with bounds not depending on $\eps$. Now let us prove estimate (\ref{w11}). Arguing as in \emph{Step 2} of the proof of Theorem \ref{thm:regv1D} we deduce that
\begin{equation}
\label{cotap}
 \int_0^T \int_0^M  |(v^p)_m| \, dm dt \leq C(T,p),
\qquad\forall p\in (1,\infty)\,,
\end{equation}
where the constant $C(T,p)$ does not depend on $\epsilon$. Note that
$$
\int_0^M  |(v^p)_m| \, dm = \int_0^M p v^{p-1} |v_m|\ dm \ge \int_0^M  p \alpha_1^{p-1} |v_m|\ dm
$$
and hence we get $v \in L^1(0,T,BV(0,M))$.

\noindent {\em Step 2: Local Lipschitz bounds and consequences.}
\label{Best}
Recall that each approximation $v_\eps$ is smooth inside $(0,M)$. This allows to perform Bernstein-type estimates. We can repeat the computations in \emph{Step 4} of Theorem \ref{thm:regv1D} (which are even simpler this time, as we have no extra term coming from a Laplacian regularization) to learn that
\begin{equation}
\label{key}
\sup_{t\in [0,T]}\Vert w(t)\Vert_\infty\leq C(T,\phi, \Vert w(0)\Vert_\infty).
\end{equation}
Here $w= |v_m|^2\phi^2$ where $\phi\geq 0$ is smooth with compact support  $[\phi_1,\phi_2]\subset (0,M)$. Now we observe the following: if $[\phi_1,\phi_2]\cap S_v(0) = \emptyset$ then $\Vert w(0)\Vert_\infty$ can be bounded independent of $\eps$ (as we already argued that $v_{0,\eps}$ is locally Lipschitz inside $(0,M)\backslash S_v(0)$ with bounds independent of $\eps$). Being $S_v(0)$ a discrete set of points, consequences are twofold:
\begin{itemize}
\item $S_{v_\eps}(t_2)\subset S_{v_\eps}(t_1)$ for any $t_2 >t_1\ge 0$ (and in particular for $t_1=0$, so that $S_{v_\eps}(t)\subset S_{v}(0)$ for any $t>0$).
\item $v_{\eps}$ is locally Lipschitz inside $(0,M)\backslash S_v(0)$ with bounds independent of $\eps$, for each $t \in [0,T]$. In fact, each $v_\eps$ has uniform (in $\eps$) interior bounds for any space and time derivative in $(0,T)\times ((0,M)\backslash S_{v}(0))$ (as a consequence of the Lipschitz bounds together with Theorem 3.1 in \cite{Ladyparabolico}, Chapter V). 
\end{itemize}

\noindent {\em Step 3: Passing to the limit as $\eps \to 0^+$.}
We observe that the regularity bounds on $v_\eps$ derived in the previous step allow to pass to the limit $\eps \to 0^+$ to some function $v$. In fact, the convergence of $v_\eps$ to $v$ is locally uniform on $(0,T)\times ((0,M)\backslash S_v(0))$ and the same goes for any derivative of the solution. Thus, $v$ satisfies the estimates of points (\ref{nmv})--(\ref{w11}) in the statement of the Theorem.  Moreover, as every $v_\eps$ satisfies the boundary conditions \eqref{verticalBC0}, so does $v$ thanks to Lemma 10 in \cite{Dirichlet}. We may show that it satisfies \eqref{RHEdv} also arguing as in \emph{Step 7} in the proof of Theorem 2.1 in \cite{Carrillo}.
\end{proof}
Now we can pass again to the original formulation to recover the entropy solution. In fact, we are able to show that, loosely speaking, the regularity of the solution $u$ cannot be worse than that of the initial datum (i.e. the number of ``singularities'' cannot increase). A regularization effect takes place also, turning Lipschitz corners into smooth points. These are consequences of the following result.
\begin{proposition}
\label{pr32}
Let
$u_0\in BV(\R)$ with $u_0(x) \geq \kappa > 0$ for $x\in [a,b]$, and $u_0(x)=0$
for $x\not\in [a,b]$. Assume that $u_0$ is locally Lipschitz in its support out of a finite set $\varphi(0,S_v(0))$. Then the entropy solution $u$ of \eqref{problema} is recovered in terms of the function $v$ constructed in Theorem \ref{extendedMDF} by virtue of \eqref{rule2}--\eqref{rule1} --extending by zero out of $[a(t),b(t)]:=[a-ct,b+ct]$. This solution satisfies the following additional properties:
\begin{itemize}

\item  $u(t)\in W^{1,\infty}_{loc}((a(t),b(t))\backslash \varphi(t,S_v(t)))$ for every $t\in (0,T)$.
\item $u(t)$ is smooth in $(a(t),b(t))\backslash \varphi(t,S_v(t))$ (in fact $u$ is smooth in $\cup_{0<t<T}(\{t\}\times ((a(t),b(t))\backslash \varphi(t,S_v(t))))$).
\item $u(t)\in BV(\R)$ for every $t\in (0,T)$. Moreover, if $u_0 \in W^{1,1}(0,M)$ then $u(t)\in W^{1,1}(\R)$ for every $t\in (0,T)$.
\end{itemize}
\end{proposition}
\begin{proof}
 We can show that formula \eqref{rule2} produces an entropy solution (hence unique) of \eqref{problema} in terms of the solution $v$ of \eqref{RHEdv} just constructed. This can be done as in the proof of Proposition 2.5 in \cite{Carrillo}; having estimate (\ref{w11}) of Theorem \ref{extendedMDF} available is crucial in order to do so. Smoothness properties are transferred from $v$ to $u$ by means of \eqref{rule2}--\eqref{rule1}. Note that according to Theorem \ref{extendedMDF} we would get $u(t)\in BV(\R)$ a.e. $t\in (0,T)$, but we can use Remark \ref{remark:bv} to show that this holds in fact for every $t\in (0,T)$. 
\end{proof}
\begin{corollary}
\label{coro}
Let $u_0$ as in Proposition \ref{pr32}. Then the function $v$ constructed in Theorem \ref{extendedMDF} is such that $v(t) \in BV(0,M)$ for every $0<t<T$. 
\end{corollary}

These is also a very important consequence of what was done so far, which sheds some light into the nature of singular points. We state it in the form of a corollary. 
\begin{corollary}
\label{nomove}
Let
$u_0\in  BV(\R)$ with $u_0(x) \geq \kappa > 0$ for $x\in [a,b]$, and $u_0(x)=0$
for $x\not\in [a,b]$. Assume that $u_0$ is locally Lipschitz in its support out of a finite set $\varphi(0,S_v(0))$ and let $v_0$ be defined by \eqref{rule2}. Pick $m^*\in S_v(0)$ and define $x=x(t):=\varphi(t,m^*)\in (a(t),b(t))$. Then
$$
 \int_{a(t)}^{x(t)} u(t)\ dx = m^* \quad \mbox{and} \quad \int_{x(t)}^{b(t)} u(t)\ dx = M- m^*
$$
 as long as the singularity at $m^*$ stands.
\end{corollary}
\begin{proof}
This is a direct consequence of \eqref{rule2} and point \eqref{nmv} in Theorem \ref{extendedMDF}.
\end{proof}
\section{Traveling waves: discontinuity fronts expire at finite time}\label{TW}
We analyze in this section some qualitative properties of solutions to \eqref{parab} through comparison with a class of specific traveling wave solutions. In this way we deduce that jump discontinuities are dissolved in finite time (see Figure \ref{fig1}), no matter if they are inside the support or at the interface. In particular, initially discontinuous interfaces become continuous after a finite time.
 Hence the dual mass distribution formulation introduced in Section \ref{MDF} for \eqref{problema} does only make sense for a finite time interval.

\begin{figure}[h]
\begin{center}
 \includegraphics[width=0.7\textwidth]{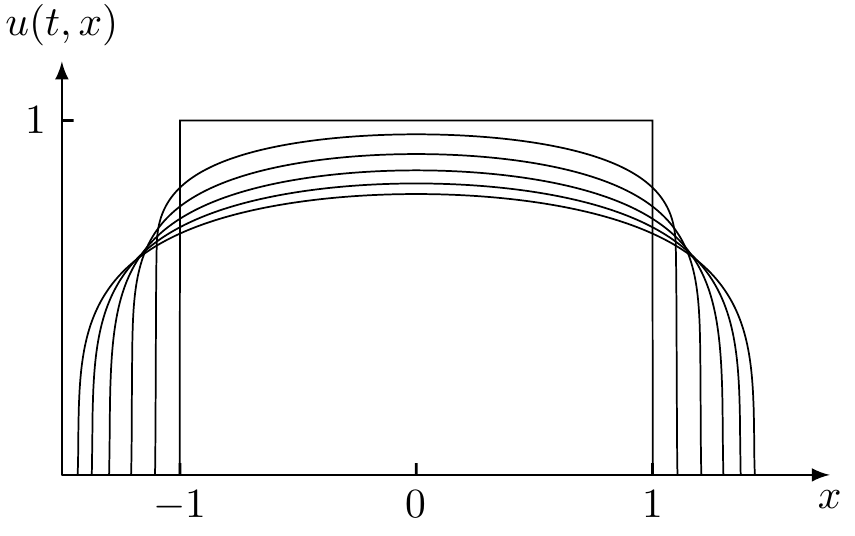}
\end{center}
\caption{Finite time dissolution of a continuous interface: Numerical time evolution by \eqref{parab} for a compactly supported initial condition with $m=4$, $\nu=c=1$ and $t \in [0,0.5] $ (the smaller the height, the more advanced the times). The time step between different profiles is $0.1$. The initial data  is the 
characteristic function of an interval. Two observations deserve to be made here: The regularization process in finite time (both in the interior and at the interfaces) as well as the different velocities for the support growth. In fact, the spreading rate is $c$ until the jump discontinuity at the interface disappears. Then this spreading rate decreases  progressively, slowed down by the super-solutions constructed in Proposition \ref{pr31}. }
\label{fig1}
\end{figure}

\begin{proposition}
\label{pr31}
Let $\sigma \in (-c,c)$ and $\xi:=x-\sigma t$. Then the  continuous function 
$u(\xi)=\left( \frac{\sigma (\xi_0 - \xi)}{\nu\sqrt{1 - (\sigma/c)^2}}\right)^\frac{1}{m}$, if 
$\sigma (\xi_0 - \xi) \geq 0$, and $u(\xi) = 0$ elsewhere, is a distributional solutions of traveling wave type
 to \eqref{parab}, for any $\xi_0 \in \R$.
\end{proposition}

\begin{proof}
A profile $u(\xi)$ is a classical traveling wave solution $u(t,x) = u(x-\sigma t)$ to $(\ref{parab})$ if 
it verifies
\begin{equation}\label{TWeq}
 \nu \, \left( \frac{ u (u^m)'}{\sqrt{1 +
 \frac{\nu^2}{c^2} (u^m)'^2 }} \right)'
 + \sigma u' = 0\, .
\end{equation}

This implies that 
\begin{equation}\label{labelK}
 \nu \,  \frac{ u (u^m)'}{\sqrt{1 +
 \frac{\nu^2}{c^2} (u^m)'^2 }} 
 + \sigma u = K
\end{equation}
for some $K\in \R$. When $K= 0$ we readily check that $u(\xi)=\left( \frac{\sigma (\xi_0 - \xi)}{\nu\sqrt{1 - (\sigma/c)^2}}\right)^\frac{1}{m}$ is a positive solution if $\sigma (\xi_0 - \xi) \geq 0$.
The matching of this positive branch with the zero solution for $\sigma (\xi_0 - \xi) \geq 0$ constitutes 
a distributional solution to $(\ref{parab})$.
\end{proof}

\begin{remark} It can be shown that $K=0$ in (\ref{labelK}) is the 
only choice that yields non-negative distributional solutions.
If we consider $r=\frac{ (u^m)'}{\sqrt{1 +
 \frac{\nu^2}{c^2} (u^m)'^2 }}$, then the planar system
$$\left\{\begin{array}{l}
u'  = \frac{1}{mu^{m-1}} \frac{r}{\sqrt{1-\frac{\nu^2}{c^2}r^2}}
\\
r' = -\frac{u'}{u} \left( \frac{\sigma}{\nu} +r \right)
\end{array} \right.$$
is equivalent to (\ref{TWeq}). Let us observe that by definition the signs of $r$ and $u'$ coincide and also
$| r| < \frac{c}{\nu}$. 
Then, by considering a graph formulation of this system, valid as long as $u$ is monotone, we have that
$$
\frac{d r}{du} = -\frac{1}{u} \left(\frac{\sigma}{\nu} -r\right).
$$ 
This equation has as solutions the constant $r(u)= -\frac{\sigma}{\nu}$ and also
$r(u) = -\left( \frac{\sigma}{\nu} +\frac{k}{u}\right)$ for certain values $k>0$, which are related to the choices $K=0$ 
and $K= \nu k$ respectively. The last ones do not yield distributional solutions of \eqref{parab}, since the associated function $u$ would become negative. Hence the extension of its positive part by zero does not comply with Rankine--Hugoniot's conditions.
\end{remark}

\begin{remark}
There are no traveling wave solutions for \eqref{problema} in the absence of reaction terms. See \cite{Campos} for a study about traveling wave solutions when \eqref{problema} is coupled with a reaction term of FKPP type.
\end{remark}

The existence of such a kind of solutions for this operator implies interesting consequences on
the qualitative behavior of arbitrary time dependent solutions. Note in particular that any bounded, compactly supported solution for (\ref{parab})  can be located under an appropriate
traveling wave for any $\sigma \in (-c,c)$ choosing $\xi_0$ large enough, see Theorem \ref{theocomp}. Let us develop this idea in the following results. 

\begin{lemma}
Let $0\le u_0\in BV(\R)$ be compactly supported in $[a,b]$ and let $u$ be the associated entropy solution of \eqref{parab}. If $d \in J_{u_0}$, then the ensuing discontinuous traveling front is dissolved within a finite time.
\end{lemma}
\begin{remark}
This result does not prevent the spontaneous appearance of jump discontinuities. It only states that the life span of any such jump discontinuity is finite.  
\end{remark}

\begin{proof}
To fix ideas, let us assume that the velocity of the discontinuity front is positive. Assume that $\|u_0\|_{\infty} = \alpha$. Given any $\sigma \in (0,c)$, we let
$$
v:= \frac{\alpha^m \nu\sqrt{1 - (\sigma/c)^2}}{\sigma}+ b.
$$
Then, according to Proposition \ref{pr31} (use $\xi_0= b+ \alpha^m \nu\sqrt{1 - (\sigma/c)^2}/\sigma$), the traveling wave profile 
 $$u_\sigma(t,x)=\left( \alpha^m +\frac{\sigma \big(b - x +\sigma t\big)}{\nu\sqrt{1 - (\sigma/c)^2}}\right)^\frac{1}{m} \chi_{(-\infty, v+ ct)}
 $$
qualifies as super-solution in the sense of Definition \ref{subsuper}; note that this is a decreasing profile such that $u_\sigma(0,b) = \alpha$ and hence $u_\sigma(0)\ge u_0$. Thus, by a comparison principle 
(see Theorem \ref{theocomp}) 
the support of $u$ must be contained in the support of any of these traveling waves for every $t>0$. Apart from this, we can use Proposition \ref{p4} to deduce that if the discontinuity persists forever then the support of 
$u$ contains the interval 
$$
\big(a,  d+c t \big), \mbox{ for any } t \geq 0.
$$
The vanishing of the discontinuity follows from the previous considerations, since $\sigma < c$ and 
\begin{equation}\label{supportbound}
d+c t
\leq
 \frac{\alpha^m \nu\sqrt{1 - (\sigma/c)^2}}{\sigma}+ \big(b  +\sigma t\big)\, .
\end{equation}
In fact (\ref{supportbound}) determines an upper bound on the time of existence for the discontinuity front, namely 
$$
t(\sigma) := \inf_{0<\sigma<c} \frac{\frac{\alpha^m \nu\sqrt{1 - (\sigma/c)^2}}{\sigma}+ b  -d}{c-\sigma}.
$$
\end{proof}

\begin{corollary}
\label{finitecontact}
Let $u_0$ be compactly supported in $[a,b]$ and such that $u_0\in BV(\R)$. Assume that $u_0(x)\ge \alpha >0$ for every $x \in [a,b]$. Let $u$ be the associated entropy solution of \eqref{parab}. Then there exists some $T^*>0$ such that 
$$
u(T^*,(a-ct)^+)=0\quad \mbox{and/or}\quad u(T^*,(b+ct)^-)=0.
$$
\end{corollary}

\section{Sub-and super-solutions: waiting time for support growth}\label{SS}

This section is devoted  to prove the existence of a waiting time for the support growth of certain compactly supported solutions to  
(\ref{parab}). In agreement  with the numerical results shown in Figures \ref{fi2} \& \ref{fi3}, this effect can be justified if certain decay conditions at the boundaries (depending on $m$) 
are verified by the  initial conditions. 
This will be a consequence of the existence of certain super--solutions with constant
support, as we state below. Some of the results in this section have been independently discovered in \cite{Giacomelli}.

\begin{figure}[h]
\begin{center}
 \includegraphics[width=\textwidth]{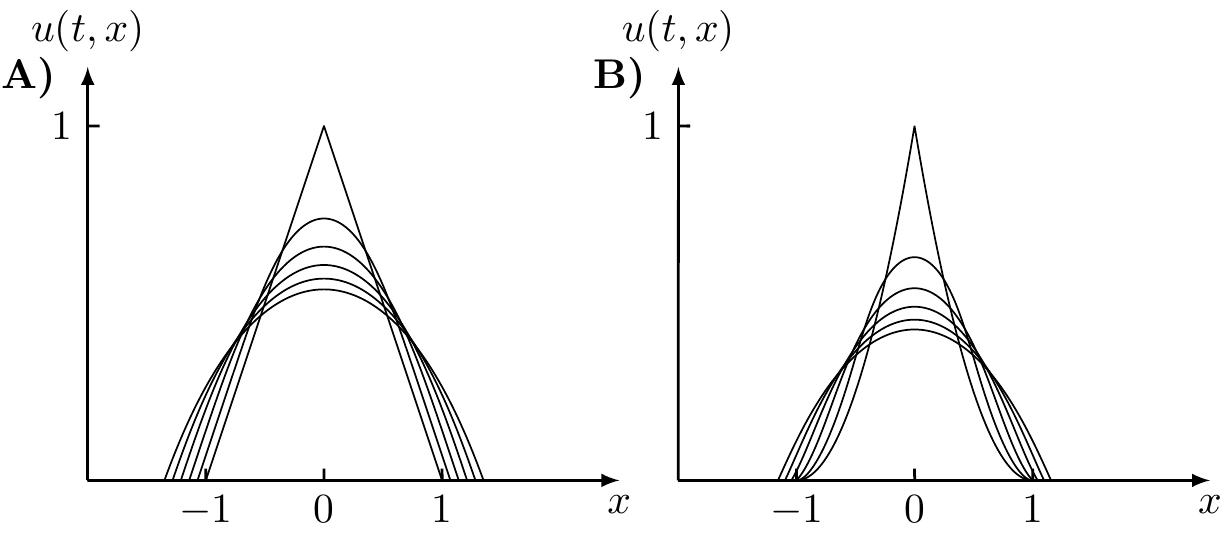}
\end{center}
\caption{ \label{fig:EME=1} Waiting time: Numerical time evolution by \eqref{parab}  for a compactly supported initial condition with $m=1$, $\nu=c=1$ and $t \in [0,0.5] $ (the smaller the height, the more advanced the times). The time step between different profiles is $0.1$. The initial data in A) $u_0$ is a triangle and in B)  is $u_0^2$. Note that in B) there is a waiting time in which the mass is reorganized before the support starts to spread; this fact does not occur in A).}
\label{fi2}
\end{figure}

\begin{figure}[h]
\begin{center}
 \includegraphics[width=\textwidth]{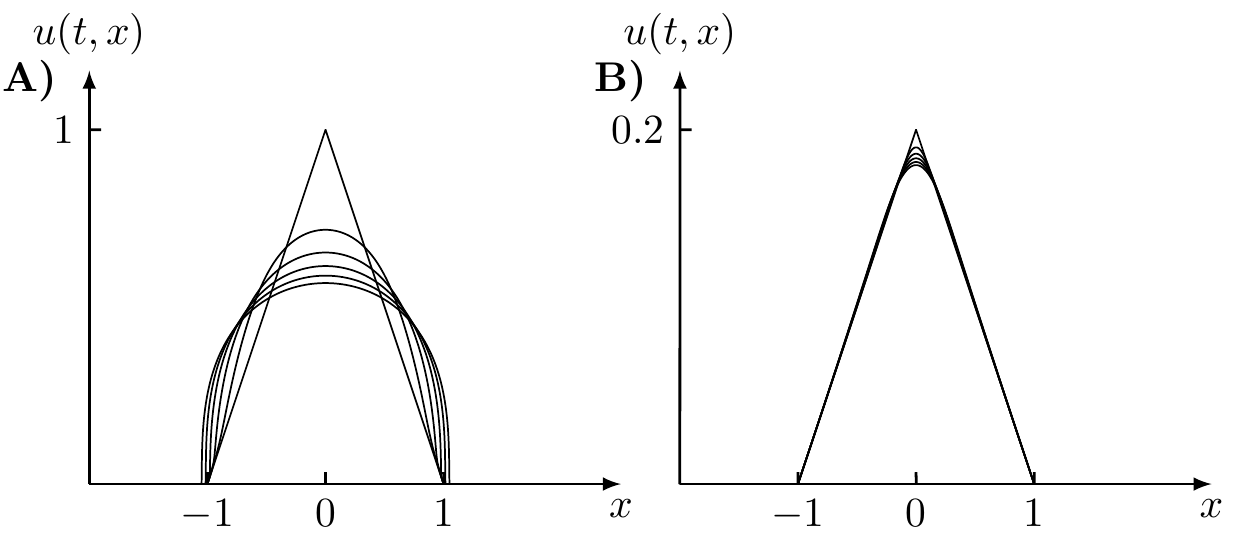}
\end{center}
\caption{Waiting time: Numerical time evolution by \eqref{parab} for a compactly supported initial condition with $m=3$, $\nu=c=1$ and $t \in [0,0.5] $ (the smaller the height, the more advanced the times). The time step between different profiles is $0.1$. The initial data are triangles of different height. Note that in B) the waiting time is longer than in A).}
\label{fi3}
\end{figure}

\begin{figure}[h]
\begin{center}
 \includegraphics[width=0.7\textwidth]{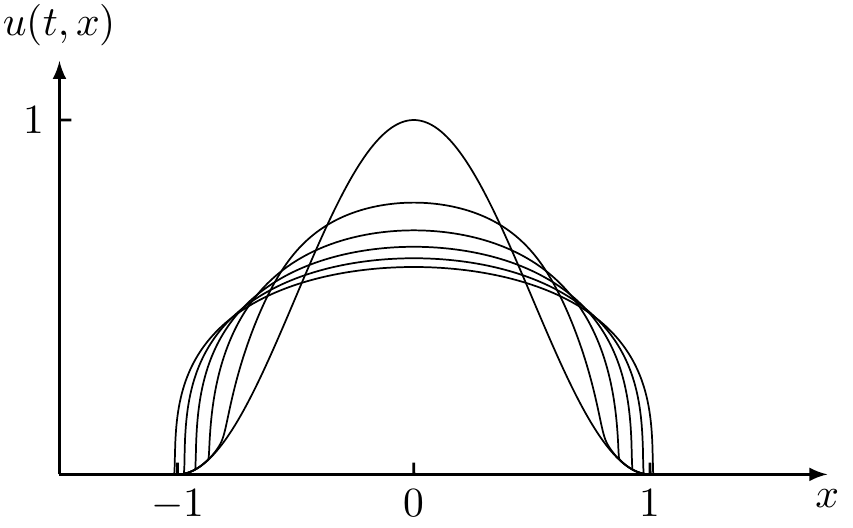}
 \end{center}
\caption{Waiting time: Numerical time evolution by \eqref{parab} for a compactly supported initial condition with $m=3$, $\nu=c=1$ and $t \in [0,0.5] $ (the smaller the height, the more advanced the times). The time step between different profiles is $0.1$. The initial data  is $u_0(x)=\cos^2(x)$ for $x\in[0,\pi]$ and zero elsewhere. Note that in the case of initial conditions with continuous interfaces, singularities in the first derivative may be developed in the interior of the support. This effect is probably justified by the fact that the solution can be sandwiched by sub- and super-solutions of the types defined in Sections \ref{TW} and \ref{SS}. }
\label{fi4}
\end{figure}

\begin{proposition}
\label{Campos_supersolution}
Let $u_0$ be such that $supp \  u_0 \subset [-\delta,\delta]$ and $0 \leq u_0 \leq v(x)^\frac1m$ for $x\in (-\delta, \delta)$, 
where  $v \in \mathrm{C}^2 (-\delta,\delta)$ is a nonnegative function verifying 
\begin{enumerate}
\item[i)] $\lim_{x \to \delta} v(x) = \lim_{x \to \delta} v'(x) = 0$\, ,
\item[ii)] $\lim_{x \to -\delta} v(x) = \lim_{x \to -\delta} v'(x) = 0$\, ,
\item[iii)] $x v'(x) < 0$, for $x \in (-\delta, \delta)-\{0\}$ ,
\item[iv)]  $v''(x) \leq k$, for $x \in (-\delta, \delta)$.
\end{enumerate} 

Then, $\tilde u(t,x) = \left(\frac{v(x)}{1-(2+m)k t}\right)^{\frac1m} $ is a super-solution of   
\eqref{parab} with initial datum $u_0$ for $t \in [0,\frac{1}{(2+m)k})$, and its support verifies 
$$\mbox{supp}(\tilde u(t,x)) =\mbox{supp}(\tilde u(0,x))\, .$$
\end{proposition}
\begin{proof}
Note that the proposed super-solution is an ansatz constructed by separation of variables,  where  the time dependent part $\alpha(t) =\left(1-(2+m)k t\right)^{-\frac1m} $ is a solution to the initial value problem
\begin{equation}\label{edoalpha}
\alpha'(t) = \frac{(2+m)k}{m} \alpha(t)^{m+1} \, , \quad \alpha(0)=1\,  ,
\end{equation}
and the space dependent part is $v(x)^\frac1m$. 

We claim that assumptions about $v$ allow us to prove that
such a function verifies
\begin{equation}\label{vinequality}
\frac{v'(x)^2}{v(x)} \leq 2k ,
\end{equation}
for any $x \in (-\delta, \delta)$.
Inequality (\ref{vinequality}) is obviously valid for $x=0$ by $iii)$. In the rest of the argument we will assume 
 $x \in (0, \delta)$. The same ideas can be analogously applied to the case $x\in (-\delta,0)$.  By using Cauchy's Mean Value Theorem we have that
 $$
 \left(v'(x)^2 - v'(y)^2\right) v'(\xi) = 2 v'(\xi)v''(\xi) \left(v(x) - v(y)\right),
 $$ 
 for any $y \in (x,\delta)$ and $\xi = \xi(x,y)$. This implies that 
  $$
 \left(v'(x)^2 - v'(y)^2\right) \leq   2 k \left(v(x) - v(y)\right) ,
 $$ 
 for any $y \in (x,\delta)$, where we have used $iv)$ and the fact that $v'(\xi) \neq 0$ due to $iii)$.   
Then, $(\ref{vinequality})$ holds by letting $y \to \delta$ and using $i)$.

Set $\Phi(s) = \frac{s}{\sqrt{1+s^2}}$. Then,  we can prove the estimate 
\begin{eqnarray}
\label{chain}
\left(\tilde u \Phi \big((\tilde u)^m\big) \right)' 
& = &\left(\alpha v^\frac1m \Phi \big(\alpha^m v'\big) \right)' 
\nonumber \\
& = &\frac{\alpha}{m} v^{\frac1m-1} v' \Phi \big(\alpha^m v'\big)+ \alpha^{m+1} v^\frac1m v'' \Phi' \big(\alpha^m v'\big)
\nonumber \\
& \leq&
\frac{\alpha^{m+1}}m v^{\frac1m-1} v'^2+ \alpha^{m+1} v^\frac1m k 
 \nonumber\\
& \leq&
\frac{(2+m)k}{m} \alpha(t)^{m+1} v^\frac1m = 
\alpha' v^\frac1m =\tilde u_t \, ,
\end{eqnarray}
where we have used $s \Phi(s) \leq s^2$, $\Phi'(s) \leq 1$, $iv)$,  inequality (\ref{vinequality}) and  equation (\ref{edoalpha}). This  concludes the proof.
\end{proof}

A similar result to Proposition \ref{Campos_supersolution} has been obtained recently and independently in \cite{Giacomelli}. Next result will be of interest  in order to ensure local separation from zero.
\begin{proposition}
\label{positivo}
Let $v$ be as in Proposition \ref{Campos_supersolution}. Then 
$$
W(t,x)= \left(\frac{1}{(2+m) k t + 1}\right)^{1/m} v(x)^{1/m}
$$
is a sub-solution with static support for any $t>0$.
\end{proposition}
\begin{proof}
The proof follows  the same lines as in Proposition \ref{Campos_supersolution}, but reversing the signs and inequalities in the chain of estimates \eqref{chain}. 
\end{proof}

\begin{remark}
Let notice that any function $v$ verifying $i)$, $ii)$ and inequality  (\ref{vinequality})  can be bounded 
by quadratic polynomials in the  following way:
$$
v(x) \leq k^2 (x-\delta)^2, \mbox{  if  } x \in [0,\delta) ,
\quad \mbox{and} \quad  
v(x) \leq k^2 (x+\delta)^2, \mbox{  if  } x \in (-\delta,0] \, .
$$
Given any function $u_0\in \mathrm{L}^\infty([-\delta,\delta])$ such that
$$\frac{u_0(x)}{ (x-\delta)^{\frac2m}} \, , \quad \frac{u_0(x)}{(x + \delta)^{\frac2m}} \in  \mathrm{L}^\infty([-\delta,\delta])\,,
$$
this allow us to assure the existence of a (maybe not optimal)  function $v$ such that Proposition \ref{Campos_supersolution} applies. Just note that the function $v(x) = \tilde k (\delta-x)^2(x+\delta)^2$ verifies the hypotheses of Proposition \ref{Campos_supersolution} for some constant $\tilde k$ large enough.
\end{remark}

Now, a simple application of our previous result to any compactly supported initial condition with appropriate 
decay estimates at the boundary allows us to conclude that the spatial support is confined to a fixed spatial interval during a certain time period.
In those cases in which  the initial support coincides with this spatial interval, we conclude that the support does not grow for a while. That is, we are in the presence of a waiting time mechanism. 
\begin{corollary}
Let $0\le u_0\in \mathrm{L}^\infty(\R)$ supported in $[a,b]$ and such that
$$\frac{u_0(x)}{ (x-a)^{\frac2m}} \, , \quad \frac{u_0(x)}{(b-x)^{\frac2m}} \in  \mathrm{L}^\infty(a,b)\, .
$$
Let $u$ be the associated entropy solution of \eqref{parab}. Then there exists some positive constant $\tilde k$ such that 
 $$
 \supp (u(t,\cdot)) \subset [a,b], \quad  \mbox{for any}\quad t \le \frac{1}{2 (2+m) \tilde k (b-a)^2 }.
 $$
\end{corollary}
\begin{proof}
We can deduce easily from the hypothesis on $u_0$ the existence of a constant $\tilde k$ such that  
$u_0 \leq (\tilde k (a-x)^2(b-x)^2)^{\frac1m}$.  Now,
we apply Proposition \ref{Campos_supersolution} to $u_0( x + (a+b)/2)$, which is  
compactly supported on $[-(b-a)/2,(b-a)/2]$ and bounded by 
$v(x)=(\tilde k (x+(b-a)/2)^2(x-(b-a)/2)^2)^{\frac1m}$. Note that $v$ verifies all the required hypotheses with $k=2 \tilde k (b-a)^2$. This clearly concludes the proof of the 
result, given the translation invariant character of (\ref{parab}) and the comparison principle in Theorem \ref{theocomp}. 
\end{proof}

\section{Smoothing effects for the relativistic heat equation: the case of isolated singularities}
\label{caseIso}

\subsection{Analysis of a model case}
\label{CasoModelo}
The aim of this section is to show the following: Given an initial datum $u_0$ with a single jump discontinuity inside its support, we can ensure under some technical conditions that there is some $t^*<\infty$ such that the associated entropy solution $u(t)$ of \eqref{problema} is smooth inside its support for every $t>t^*$. This means that an isolated jump discontinuity is dissolved in finite time and after that the solution is smooth everywhere inside the support. The analysis of this simple case will allow to show in Section \ref{generaliza}, via adequate reduction to simpler cases,  that the regularizing effect of \eqref{problema} is indeed much more general than what we will discuss here. 
 
To be more precise, in this Section we will track the evolution of initial data which are compactly supported in an interval, being both interfaces discontinuous and having another jump discontinuity inside their support. 

\begin{definition}
\label{clase0}
Let $u_0 \in L^\infty(\R)$. We say that $u_0 \in \mathcal{J}_0$ if the following conditions hold:
\begin{enumerate}
\item $u_0$ is supported in $[a,b]$. \label{conexo}
\item $u_0\in BV(\R)$
\item $u_0\ge \kappa>0 $ for $x \in [a,b]$. \label{sepzero}
\item The jump set of the initial datum is $J_{u_0} = \{a,\delta,b \}$, with $a<\delta<b$. Let us assume that the discontinuity at $\delta$ will travel to the right (say), i.e. we choose $\nu^{\delta}=+1$ and so $u^+(\delta)<u^-(\delta)$. \label{solounsalto}
\item $u_0 \in W^{2,1}(\R \backslash J_{u_0})$ (hence $u_0 \in W^{1,\infty}(\R \backslash J_{u_0})$). \label{rmeas}
\item $(u_0)_x(\delta)^-,(u_0)_x(\delta)^+\le 0$. \label{rmeas2}
\end{enumerate}
\end{definition}
{Some comments are in order here. First, it is mandatory to ensure that $u \in BV_{loc}(Q_T)$ in order to use Proposition \ref{p4}, which is crucial in what follows. To achieve this, Lemma \ref{l3} is the only tool so far. This is why we require of (\ref{rmeas})-(\ref{rmeas2}). And second, (\ref{conexo}), (\ref{sepzero}) and (\ref{solounsalto}) are assumed just for the sake of technical convenience and a clearer exposition; this will become clear in the sequel. We will remove these assumptions in Section \ref{generaliza}.}

We let $u(t)=u(t, \cdot)$ and $u_t(t) = \frac{\partial u}{\partial t} (t, \cdot)$. Given $T>0$, set $Q_T:=(0,T)\times \R$. Our aim is to prove the following result:
\begin{theorem}
\label{ilustracion}
Let $u_0 \in \mathcal{J}_0 $ and let $u$ be the associated entropy solution of \eqref{problema}. Then:
\begin{enumerate}
\item $u(t) \in BV(\R)$ for each $t>0$ and $u \in BV((0,T)\times \R)$, for every $T>0$. 
\item $u(t)$ is supported on $[a-ct,b+ct]$ and $u(t)\ge \kappa(t)>0$ in the support.
\item There exists some $0<T^*<\infty$ such that $u(t)\in W^{1,1}(a-ct,b+ct)$ and $u(t)$ is smooth inside its support, for every $t\ge T^*$.
\end{enumerate}
\end{theorem}

The rest of the section constitutes a proof for the third statement of this Theorem (the others being already proved in Section \ref{sc11}). To begin with, let us state some properties of entropy solutions with
 initial data in $\mathcal{J}_0 $ that are obtained as a consequence of the results in Section \ref{sc11}.
\begin{lemma}
\label{l4}
Let $u_0 \in \mathcal{J}_0 $ and let $u$ be its associated entropy solution of \eqref{problema}. Then:
\begin{enumerate}
\item  $(a(t),b(t))=(a-ct,b+ct)$.
\item $u(t) >\kappa(t)>0$ for $x \in (a(t),b(t))$. 

\item Jump discontinuities at the interfaces $x=a(t)$ and $x=b(t)$ are not dissolved in finite time.
\item $u \in BV([\tau,T]\times \R) $ for any $\tau>0$ and $u_t(t)$ is a finite Radon measure in $\R$ for any $t>0$, so Proposition \ref{p4}  and Remark \ref{r2} hold.
\end{enumerate}
\end{lemma}

Using the previous result we can show that  the traces of the flux can be  computed in a stronger sense than the one in \cite{leysalto}. This is the content of the next statement, which we formulate in a broader context.
\begin{lemma}
\label{milestone}
Let $u_0 \in BV(\R)^+$ and let $u$ the associated entropy solution of \eqref{problema} in $Q_T$. Assume that $u_0$ is supported in $(a,b)$ and $u_0(x)>\kappa>0$ for every $x \in (a,b)$. Assume further that $u_t(t)$ is a finite Radon measure in $\R$, for any $t>0$. Then:
\begin{itemize}
\item[i)] $\a(u,u_x),\ \b(u,u_x) \in BV(\R)$, where $z \b(z,\xi)=\a(z,\xi)$, for every $t>0$.
\item[ii)] $[\a(u,u_x)\cdot \nu^\Omega]=u_{|\partial \Omega} \, [\b(u,u_x)\cdot \nu^\Omega]$ for every $x \in \partial \Omega$, for every subdomain $\Omega \subset (a(t),b(t))$ and for every $t>0$.
\item[iii)]  Rankine--Hugoniot's condition \eqref{slopeconfig} is verified. In fact, for every $(t_0,x_0)\in J_u$ and for every spatiotemporal ball $B$ about $(t_0,x_0)$ which is contained in $\cup_{t>0}(a(t),b(t))$, there holds that
$$
[\b\cdot \nu^{J_{u(t)}}]_+=c\quad \mbox{and}\quad [\b\cdot \nu^{J_{u(t)}}]_-=c\quad \mbox{for every}\ (t,x)\in J_u\cap B.
$$

\end{itemize}
\end{lemma}
\begin{proof}
The fact that $\a(u,u_x) \in BV(\R)$ for every $t>0$ is given in Remark \ref{r2}, while $\b(u,u_x) \in BV(\R)$ for every $t>0$ follows from Lemma 5.5 in \cite{leysalto}. The factorization of the trace follows from Lemma 5.6 in \cite{leysalto}. Proposition 8.1 in \cite{leysalto} together with the remarks about vertical contact angles that are stated afterwards constitute a proof for the third statement. 
\end{proof}
\begin{corollary}
\label{dual_milestone}
Let $u_0$ be as in Lemma \ref{milestone}. Consider the associated function $v_0$ defined by \eqref{rule2}--\eqref{rule1} and assume that $v_0$ is regular enough so that Theorem \ref{extendedMDF} applies. Then:
\begin{itemize}
\item $\frac{\nu v_m}{\sqrt{v^4+\frac{\nu^2}{c^2}(v_m)^2}} \in BV(0,M)$  for every $t>0$.
\item For every $t>0$, the following identities are verified
$$
\frac{\nu v_m}{\sqrt{v^4+\frac{\nu^2}{c^2}(v_m)^2}}(t,0^+)= \frac{\nu v_m}{\sqrt{v^4+\frac{\nu^2}{c^2}(v_m)^2}}(t,M^-)=c. 
$$
\end{itemize}
\end{corollary}
The following pair of results are also easy consequences of Lemma \ref{milestone}.
\begin{lemma}
Let $T>0$ and $\lambda,\mu \in \R$. Assume that $u$ solves \eqref{problema} in $Q_T$ and is smooth in $\cup_{0<t<T}(a-\lambda t,b+\mu t)$. Then, for any $0<t<T$,
\begin{equation}
\label{masscomp}
\frac{d}{dt}\int_{a-\lambda t}^{b+\mu t} u\, dx = u(t,b+\mu t) \left(\mu + \b(t,b+\mu t)\right) - u(t,a-\lambda t) \left(\lambda +\b(t,a-\lambda t)\right).
\end{equation}
\end{lemma} 
\begin{lemma}
\label{handy}
Assume that $u(t)$ is smooth in $(a(t),\tilde{\delta}(t))$ and in $(\tilde{\delta}(t),b(t))$ for $0\le t<T^*$. Provided that
$$
 t \mapsto \int_{a(t)}^{\tilde{\delta}(t)} u(t)\ dx \quad \mbox{and} \quad t \mapsto \int_{\tilde{\delta}(t)}^{b(t)} u(t)\ dx
$$
are constant functions for $0\le t<T^*$, the following assertions are verified:
\begin{itemize}
\item $t \mapsto \tilde{\delta}(t)$ can be differentiated indeed, for any $0< t<T^*$.
\item $\b(u,u_x)(t,\tilde{\delta}(t)^-)$ and  $\b(u,u_x)(t,\tilde{\delta}(t)^+)$ agree, for every $t \in (0,T^*)$. 
\end{itemize}
\end{lemma}

Thanks to Proposition \ref{pr32} we have an alternative description of $u$ in terms of a globally defined function $v: (0,T)\times (0,M) \rightarrow \R^+$. In such a way, we know that no new singularities will appear: Letting $\delta=\tilde{\delta}(0)$ then $u(t)$ is smooth in $(a(t),b(t))\backslash \varphi(t,\varphi^{-1}(0,\delta))$ (that is, everywhere in its support but maybe on the trajectory traced out by the jump discontinuity). Thus, our first step in order to prove Theorem \ref{ilustracion} is to analyze the behavior of the jump discontinuity at $x={\delta}$ more closely. We have some information already coming from Rankine--Hugoniot and entropy conditions:
\begin{lemmad}
\label{l5}
Let $u_0 \in \mathcal{J}_0 $ and let $u$ be its associated entropy solution. Then the jump discontinuity at $x=\delta$ is not dissolved instantaneously, \emph{i.e.}, there exists some $t_1>0$ such that $\mathcal{J}_{u(t)}$ contains precisely three elements for every $t<t_1$.

Let us take $t_1$ the maximal time with that property (that is $t_1 :=\min \{t/\#  \mathcal{J}_{u(t)} \neq 3Ê\}$). The initial jump discontinuity at $x=\delta$ will be traveling to the right with speed $c$, for $t<t_1$. Let us denote $\delta(t):=\delta+ct$ the virtual trajectory of the base point of the discontinuity. Then condition \eqref{slopeconfig} is verified for $0<t<t_1$.
\end{lemmad}
\begin{proof}
Since $u \in C([0,T],L^1(\R))$, we can find a sequence $t_n\searrow 0$ such that $\{u(t_n)\}$ converges a.e. $x\in \R$. This is not compatible with an instantaneous dissolution of the jump discontinuity. The trajectory of its base point is determined by Rankine--Hugoniot conditions, while the last statement follows from Lemma \ref{milestone} iii).
\end{proof}
 Let us introduce 
$$
m_l:=\int_a^l u_0 \ dx \quad \mbox{and}\quadÊm_r:=\int_l^b u_0 \ dx= M-m_l.
$$ 
Recall that $v(t) \in BV(0,M)$ for every $t>0$; this allows us to compute traces at $m_l$ for any $t>0$. A variant of Corollary \ref{dual_milestone} shows that
\begin{equation}
\label{innerbc}
\left\{\begin{array}{ll}
\frac{\nu v_m}{\sqrt{(v)^4+\frac{\nu^2}{c^2}(v_m)^2}}(t,m_l^-) =  c, &{}\\
 {} & \qquad \forall \, 0<t<t_1.\\
 \frac{\nu v_m}{\sqrt{(v)^4+\frac{\nu^2}{c^2}(v_m)^2}}(t,m_l^+) =  c,&{}\\
\end{array}\right.
\end{equation}
Nothing precludes that \eqref{innerbc} may hold true past $t_1$.
\begin{lemmad}
\label{LD1}
The following statements hold true:
\begin{enumerate}
\item We have that $S_v(t)=S_v(0)$ for every $t<t_1$. Let us take $t^*\in [t_1,\infty]$ the maximal time with this property (i.e the first time at which the singularity vanishes, $t^*:= \min \{t/ S_v(t)=\emptyset\},\ t^*=+\infty$ if the former set is empty). 
\item If $t^*>t_1$ we extend $\delta(t)$ to $(0,t^*)$ as $\delta(t)=\varphi(t,\varphi^{-1}(0,\delta))$. Then both 
$$
 t \mapsto \int_{a(t)}^{{\delta}(t)} u(t)\ dx \quad \mbox{and} \quad t \mapsto \int_{{\delta}(t)}^{b(t)} u(t)\ dx
$$
are constant functions for $t<t^*$. \label{2point}
\item Let $t_2$ be the maximal time such that \eqref{innerbc} holds true (that is the first time at which \eqref{innerbc} is violated, $t_2 := \min \{t, \mbox{ such that } \eqref{innerbc} \,Ê\mbox{does not hold}Ê\}, \ t_2=+\infty$ if the former set is empty). Then $t_2=t^*$.
\end{enumerate}
\end{lemmad}
\begin{proof}
The first statement is clearly deduced from \eqref{rule1}. The second is just Corollary \ref{nomove}. To prove the third, we notice that $t_2 \le t^*$ by definition. Now let us consider what happens with \eqref{innerbc} at $t=t_2$. As $v_t(t_2)$ is a finite Radon measure on $(0,M)$, spatial traces of the flux are defined for $t=t_2$ and any $m\in [0,M]$. Then, either one of the lateral traces in \eqref{innerbc} becomes different from $c$ or both lateral traces differ from $c$ at the same time. Given that $v(t_2)$ is smooth and bounded in $(0,m_l) \cup (m_l,M)$, we deduce in the second case that $v_m\in L_{loc}^\infty(0,M)$. Hence $m \mapsto v(t_2,m)$ is Lipschitz continuous or even smooth at $m=m_l$. Thus $S_v(t^*)=\emptyset$ and $t_2=t^*$ in this case.

Let us show that the first case leads to contradiction. In that case, we would have $S_v(t_2)\neq\emptyset$, thus $t_2<t^*$. Using point (\ref{2point}) of the present result, no mass flow is allowed across $m_l$ for any $t \in [t_2,t^*)$. Then Lemma \ref{handy} applies,  giving a contradiction that concludes the proof. 
\end{proof}
We are now ready to apply the change of variables studied in Section \ref{MDF} in the regions $(a(t),\delta(t))$ and $(\delta(t),b(t))$ separately. To that end, we consider a pair of functions $v^l(t,m),\, v^r(t,m)$ defined for $t<t^*$,
$$
v^l(t,\cdot) :  (0,m_l) \rightarrow (a(t),\delta(t)),\quad v^r(t,\cdot) :  (0,m_r) \rightarrow (\delta(t),b(t)),
$$
together with the following problems:
\begin{equation}
\label{vi}
v_t^l = \left(\frac{\nu v_m^l}{\sqrt{(v^l)^4+\frac{\nu^2}{c^2}(v_m^l)^2}}\right)_m,\quad m \in (0,m_l),\ t \in (t,t^*) 
\end{equation}
with boundary conditions
\begin{equation}
\label{vibon}
\frac{\nu v_m^l}{\sqrt{(v^l)^4+\frac{\nu^2}{c^2}(v_m^l)^2}} n =  c,\quad m \in \partial (0,m_l),\quad n(0)=-1,\ n(m_l)=1,
\end{equation}
and 
\begin{equation}
\label{vibis}
v_t^r = \left(\frac{\nu v_m^r}{\sqrt{(v^r)^4+\frac{\nu^2}{c^2}(v_m^r)^2}}\right)_m,\quad m \in (0,m_r),\ t \in (t,t^*) 
\end{equation}
with boundary conditions
\begin{equation}
\label{vibonbis}
\frac{\nu v_m^r}{\sqrt{(v^r)^4+\frac{\nu^2}{c^2}(v_m^r)^2}} n = c\ \mbox{at}\ m=m_r\ \mbox{and}\ \frac{\nu v_m^r}{\sqrt{(v^r)^4+\frac{\nu^2}{c^2}(v_m^r)^2}} n = c \ \mbox{at}\ m=0,
\end{equation}
being $n(0)=+1,\ n(m_r)=1$. The boundary conditions that we impose here are the natural ones after Lemma \ref{l4} and Lemma and Definition \ref{l5} (compare with Section \ref{MDF}). Following Corollary \ref{dual_milestone}, we see that \eqref{vibon}, \eqref{vibonbis} can be interpreted as relations between traces of functions of bounded variation and there is no need to use weak traces to describe the behavior at the boundary.

Arguing as in Section \ref{MDF} we get the following results:
\begin{proposition}
\label{p11}
There exists a smooth solution $v^l$ of \eqref{vi} in $(0,t^*)\times (0,m_l)$ with $v^l(0,m)=v_0^l(m)$ and satisfying boundary conditions \eqref{vibon}. A similar result holds for \eqref{vibis}--\eqref{vibonbis}. We have that $v(t)=v^l\chi_{(0,m_l)} + v^r\chi_{(m_l,M)}$ for every $0<t<t^*$.
\end{proposition}
\begin{proposition}
\label{p12}
Let us decompose $u(t,x):= u^l \chi_{(a(t),\delta(t))}(x) + u^r \chi_{(\delta(t),r(t))}(x)$ for any $t<t^*$. Then 
\begin{enumerate}
\item $u^l$ is related to $v^l$ by means of the change of variables $\varphi$ restricted to $(0,m_l)$; $u^r$ is related to $v^r$ by means of the change of variables $\varphi$ restricted to $(m_l,M)$.
\item $u^l(t) \in W^{1,1}(a(t),\delta(t))$ and $u^r(t) \in W^{1,1}(\delta(t),b(t))$ for {every} $t \in (0,t^*)$.
\item $u^l(t) \in W^{1,\infty}_{loc}(a(t),\delta(t))$ and $u^r(t) \in W^{1,\infty}_{loc}(\delta(t),b(t))$ for every $t \in (0,t^*)$.
\item Both $u^l$ and $u^r$ are smooth in their domains of definition.
\end{enumerate}
\end{proposition}
Using this parallel formulation, we can show that the size of the inner jump at $x = \delta (t)$ cannot increase with time. More precisely: 
\begin{proposition}
\label{p13}
Let $t<t^*$. Then:
\begin{itemize}
\item $u^r(t+h,\delta(t+h)_+)\ge u^r(t,\delta(t)_+)$, for any $0<t<t+h<t^*$.
\item $ u^r(t+h,b(t+h)_-)\le u^r(t,b(t)_-)$, for any $0<t<t+h<t^*$.
\item $  u^l(t+h,\delta(t+h)_-)\le u^l(t,\delta(t)_-)$, for any $0<t<t+h<t^*$.
\item $u^l(t+h,a(t+h)_-)\le u^l(t,a(t)_-)$, for any $0<t<t+h<t^*$.
\end{itemize}
\end{proposition}
\begin{proof}
 Let us show the first statement, the proof of the rest being similar. Let $t\in (0,t^*)$ be fixed. Being $v^r$ smooth at $(0,t^*)\times (0,m_r)$, we compute for any $\lambda \in (0,m_r)$
\begin{eqnarray*}
 \frac{d}{dt} \int_0^\lambda  v^r(t,m)\ dm &=& \int_0^\lambda \frac{d}{dm}\left(\frac{\nu v_m^r(t)}{\sqrt{(v^r(t))^4+\frac{\nu^2}{c^2}(v_m^r(t))^2}} \right)\ dm 
\\
&=& -\frac{\nu v_m^r(t)}{\sqrt{(v^r(t))^4+\frac{\nu^2}{c^2}(v_m^r(t))^2}} \Bigg|_{m=\lambda^-} - c.
\end{eqnarray*}
 Thanks to the uniform estimates for $v^r$ provided by Steps \emph{1} and \emph{2} of Theorem \ref{extendedMDF}, we get to
$$
\frac{1}{\lambda} \frac{d}{dt} \int_0^\lambda  v^r(t,m)\ dm <0,
$$
for any $\lambda \le  m_r$ and $t<t^*$. Then, for any $h>0$ such that $t+h<t^*$, we can integrate in time to get
$$
\frac{1}{\lambda}  \int_0^\lambda  v^r(t+h,m)-v^r(t,m) \ dm < 0.
$$
Now, we take traces at $m=0^+$ letting $\lambda \to 0$. We conclude that
$$
v^r(t+h,0_+)-v^r(t,0_+) \le 0,
$$
for any $0<t<t+h<t^*$. This implies the final result.
\end{proof}
The previous statement shows that the size of the jump discontinuity cannot increase with time. Let us show next that it vanishes in finite time.
\begin{lemma}
\label{celeste}
We have that $t^*<\infty$ and $u(t^*,\delta(t^*)^-)=u(t^*,\delta(t^*)^+)$.
In fact, $u_x(t^*) \in L^\infty_{loc}(a(t^*),b(t^*))$. Thus $u(t^*)\in W^{1,\infty}_{loc}(a(t^*),b(t^*))$. 
\end{lemma}
\begin{proof}
 Assume first that $u(t,\delta(t)^-)>u(t,\delta(t)^+)$ for every $t>0$. Then $u(t,\delta(t)^+)\ge u_0(\delta^+)$, for every $t>0$, in particular as a consequence of Proposition \ref{p13}. This contradicts Proposition \ref{concdecay2}. 

Thus, there exists some $t_3<\infty$ such that $u(t_3,\delta(t_3)^-)\le u(t_3,\delta(t_3)^+)$. Set $t_3:= \min \{t/Êu(t,\delta(t)^-)\le u(t,\delta(t)^+)\}$. Note that $t_3 \le t^*$ as $t^*$ is defined. In fact $u(t_3,\delta(t_3)^-)= u(t_3,\delta(t_3)^+)$, otherwise the fact that $u \in C([0,T],L^1(\R))$ would be violated. 

Hence, there exists some $t_3\le t^*$ such that $u(t_3,\delta(t_3)^-)= u(t_3,\delta(t_3)^+)$. As long as $t<t^*$, boundary conditions \eqref{innerbc} hold true. Thus, Proposition \ref{p13} applies and we deduce that $\mathcal{J}_{u(t)}=\{a(t),b(t)\}$ for $t_3 \le t <t^*$. Thanks to Theorem \ref{extendedMDF} and Proposition \ref{pr32} we deduce that our solution has $W^{1,1}$ spatial regularity inside the support from $t_3$ on and moreover it is smooth out of $x = \varphi(t,\varphi^{-1}(0,\delta))$. 

Let us show next that $t^*<\infty$: Given that boundary conditions \eqref{innerbc} hold true for $t<t^*$, Proposition \ref{p13} applies and $\|u(t)\|_\infty \ge u_0(\delta^+)$ for $t<t^*$ as a consequence. But this would be in contradiction with Proposition \ref{concdecay2} if $t^*=+\infty$. Altogether, the first statement of the Lemma is proved.

The remaining statements follow as in the proof of Lemma and Definition \ref{LD1} (recall that $S_v(t^*)=\emptyset$ by definition).
\end{proof}
The previous result does not preclude the possibility of having $t_3<t^*$. Were that the case, then Lemma \ref{handy} would show that $\delta(t)=\delta+ct$  for $t<t^*$.  
One way or another, once we have Lemma \ref{celeste} at our disposal we may apply Proposition \ref{pr32} with $u(t^*)$ as initial datum. We conclude that $u(t)$ is smooth inside its support for every $t>t^*$. Combining all the results so far completes the proof of Theorem \ref{ilustracion}.

\subsection{Analysis of H\"older cusps, continuous interfaces and isolated zeros}
The purpose of this paragraph is to extend the  ideas involved in the proof of Theorem \ref{ilustracion} in order to treat a number of other distinctive features that may be present during the evolution given by \eqref{problema}. 
We will state and prove here several partial statements treating separately the evolution of an initial datum with a single H\"older cusp, with continuous interfaces or with an isolated zero inside its support. These results will be blended together with that of Theorem \ref{ilustracion} to conform a completely general statement in Section \ref{generaliza} below.

\subsection*{Non-Lipschitz continuity points inside the support}
We can show that there is a regularization effect which dissolves continuity points for which Lipschitz continuity does not hold (including the case of H\"older cusps):
\begin{proposition}
\label{nocusp}
Let $u_0 \in L^\infty(\R)$ such that the following conditions hold:
\begin{itemize}
\item $u_0$ is supported in $[a,b]$ and $u_0\ge \kappa>0 $ for $x \in [a,b]$.
\item $u_0\in BV(\R)$ and $J_{u_0} = \{a,b \}$.
\item $u_0 \in W^{1,\infty}_{loc}(\R \backslash (J_{u_0}\cup \{\delta\}))$, $a<\delta<b$.
\end{itemize}
Let $u$ be the entropy solution of \eqref{problema} with initial datum $u_0$. Assume that $u_t(t)$ is a finite Radon measure for any $t>0$. Then:
\begin{enumerate}
\item $u(t) \in BV(\R)$, for each $t>0$, and $u \in BV((0,T)\times \R)$, for every $T>0$. 
\item $u(t)$ is supported on $[a-ct,b+ct]$ and $u(t)\ge \kappa(t)>0$ in the support.
\item  $u(t)\in W^{1,1}(a-ct,b+ct)$, for every $t>0$. Moreover, there exists some $T^*\ge 0$ such that  $u(t)$ is smooth inside its support, for every $t >T^*$.
\end{enumerate}
\end{proposition}
\begin{proof}
Thanks to our hypothesis both lateral traces of $u_0$ at $x=\delta$ coincide. Hence $u_0$ is continuous at $x=\delta$ and so $u_0 \in W^{1,1}(a,b)$. Then we are able to use Proposition \ref{pr32}, which ensures that $u(t)\in W^{1,1}(a(t),b(t))$ for every $t>0$.

Using Theorem \ref{extendedMDF} we are able to pass to the inverse distribution formulation \eqref{RHEdv}--\eqref{verticalBC0}. Then either $v(t)$ is smoothed out instantaneously or there is some $t_1\in (0,\infty]$ such that $S_v(t)$ is not empty for every $t<t_1$. We are in the first case if, for instance, $u_0$ has a H\"older cusp at $x=\delta$ (combining Corollary \ref{nomove}, Lemma \ref{handy}) and the fact that $u \in C([0,T],L^1(\R))$.

Assume now that we are in the second case; we pick $t_1$ maximal with this property. We notice that $S_v(0)=\{Ê\varphi(0,\delta)\}$. Let $\delta(t):=\varphi(t,\varphi^{-1}(0,\delta))$. Then Corollary \ref{nomove} ensures that mass transfer across $\delta(t)$ is prevented as long as $\varphi(0,\delta)$ lies in the singularity set of $v(t)$. This can be combined with Lemma \ref{handy} to argue that
$$
\b(u,u_x)(t,{\delta}(t)^-)=\b(u,u_x)(t,{\delta}(t)^+),\quad \mbox{for any}\ t\in (0,t_1),
$$
and that both assume either the value $+c$ or $-c$.
 Thus, \eqref{innerbc} is satisfied in $(0,t_1)$ with $m_l=\varphi(0,\delta)$. Then we can argue exactly as  in Lemma and Definition \ref{l5} on. Our situation here is even simpler, as we can assume that $t_3=0$. There is just one minor change: We don't know a priori if $\delta(t)=\delta+ct$ or $\delta(t)=\delta-ct$. Apart from that, mimicking those arguments we show that there is a regularizing effect on the long time run.
\end{proof}

 \subsection*{Continuous interfaces} 
 Now we show that the statement of Theorem \ref{ilustracion} remains true if we substitute discontinuous interfaces by continuous ones. In fact, we can argue like in Proposition 3.2 of \cite{Carrillo}, as long as we are separated from zero inside the support. Let us assume for instance that both interfaces are continuous. 
\begin{definition}
Let $u_0 \in L^\infty(\R)$. We say that $u_0 \in \mathcal{J}_C$ if the following conditions hold:
\begin{enumerate}
\item $u_0$ is supported in $[a,b]$ and $u_0>0 $ for $x \in (a,b)$.
\item $u_0\in BV(\R)$
\item The jump set of the initial datum is $J_{u_0} = \{\delta \}$, with $a<\delta<b$. Let us assume that the discontinuity at $\delta$ will travel to the right (for instance), i.e. we choose $\nu^{\delta}=+1$ and so $u^+(\delta)<u^-(\delta)$.
\item $u_0 \in W^{1,\infty}(\R\backslash J_{u_0})$ and $u_0(x) \to 0$ as $x \to a,b$.
\item $u_0 \in W^{2,1}(\R \backslash J_{u_0})$ and $(u_0)_x^-(\delta),(u_0)_x(\delta)^+\le 0$.
\end{enumerate}
\end{definition}
\begin{proposition}
\label{pr52}
The results in Theorem \ref{ilustracion} hold true for $u_0 \in \mathcal{J}_C$, with the following exception: the property $u(t)>0$ holds only in the interior of the support. Moreover, if
$u_0(x) \leq A (b-x)^\alpha(x-a)^\alpha$ for some $A, \alpha > 0$,
then $u(t,x) \leq A(t) (b(t)-x)^\alpha(x-a(t))^\alpha$ for any $x\in
(a(t),b(t))$, $t > 0$ and some $A(t)$. In that case, $u(t,x)$ is a
continuous function in a neighborhood of the interface that tends to $0$ as $x \to a(t),b(t)$.
\end{proposition}
\begin{proof}
Theorem \ref{th2} ensures that the support evolves as $[a-ct,b+ct]$ and that we are separated from zero inside it. Then we can obtain a number of statements resembling those in Section \ref{MDF} but only of local nature.
 The point here is that $v_0\notin L^\infty$, because it diverges at $\partial (0,M)$. This can be bypassed as in the proof of Proposition 3.2 of \cite{Carrillo}: modify the initial datum adding a constant $\delta$ that will converge to zero afterwards. In this way we obtain regularized solutions $v_\delta$ which are bounded, on which we can perform  estimates like those in Theorem \ref{extendedMDF}. {This time also the integral estimates will be local (in order to avoid the lack of integrability at the boundary). But such local bounds suffice to pass to the limit in $\delta \to 0$ and construct a suitable entropy solution, as explained in the proof of Proposition 3.2 of \cite{Carrillo}}. This means that we can handle inverse distribution formulations in terms of $v,\,v^l$ and $v^r$ as we did in Subsection \ref{CasoModelo}. 

{Let us detail what would be the minor changes.} First, we must substitute $BV(0,M)$ by $BV_{loc}(0,M)$ in point (\ref{w11}) of Theorem \ref{extendedMDF}. 
And second, Corollary \ref{coro} only asserts $BV_{loc}(0,M)$-regularity this time; we cannot compute traces on $\partial (0,M)$. We would have (say) that $v^l\in BV(m_l/3,m_l)\cap BV_{loc}(0,2 m_l/3)$ and something similar for $v^r$; this is more that enough in order to proceed. Taking these remarks into account, everything goes as in Subsection \ref{CasoModelo}. Moreover, the super-solutions given in Lemma \ref{supercont} provide us with information on the behavior of the interfaces, as quoted in the statement.
\end{proof}

It is clear that when there is just one continuous interface the arguments can be performed in the same way; the only significative difference is maybe that the super-solutions in Lemma \ref{supercont} can be used to control only one end of the support. It is also clear that we can get a variant of Proposition \ref{nocusp} with continuous interfaces at one or both ends.

 \subsection*{Analysis of isolated zeros}
 It would seem that the presence of isolated zeros inside the support could spoil the passage to the inverse distribution formulation. Let us examine more closely the dynamical behavior of such isolated zeros. The following statement is our main tool in that regard.
\begin{proposition}
\label{wprofile}
Given $R_0,\alpha_0, l,\kappa >0$, there are values $\beta_1,\beta_2 >0$ large enough such that
$$
w(t,x)=
\exp \{Ê-\beta_1 t - \beta_2 t^2\} \alpha_0 \frac{c}{\nu} \Theta(t,x),
$$
being 
\begin{eqnarray*}
\Theta(t,x)&=&\Bigg\{\sqrt{(\kappa + ct)^2-|x+l|^2} \chi_{(-l-\kappa-ct,\min \{0,-l+\kappa+ct\}]} 
\\ &&+ \sqrt{(\kappa + ct)^2-|x-l|^2} \chi_{(\max \{0,l-\kappa-ct\},l+\kappa+ct)}
\Bigg\}, 
\end{eqnarray*}
is an entropy sub-solution of \eqref{problema}.
\end{proposition}
\begin{proof}
The above profile represents two configurations like the one in Proposition 2 of \cite{ARMAsupp}, each of them with initial radius $\kappa$ and centered at $\pm l$, so that the arrangement is symmetric around the origin. Thus, as
long as $l-\kappa-ct>0$ the proof given in \cite{ARMAsupp} does the job. We only have to modify it slightly for $t_0\ge (l-\kappa)/c$ in order to get our statement. For that, let
$$
D_l (t):= - \exp \{Ê-\beta_1 t - \beta_2 t^2\} \alpha_0 \frac{c}{\nu} \frac{x+l}{\sqrt{(\kappa + ct)^2-|x+l|^2}} \chi_{(-l-\kappa-ct,\min \{0,-l+\kappa+ct\}]}  \mathcal{L}^1
$$
and
$$
D_r (t):= - \exp \{Ê-\beta_1 t - \beta_2 t^2\} \alpha_0 \frac{c}{\nu} \frac{x-l}{\sqrt{(\kappa + ct)^2-|x-l|^2}}  \chi_{(\max \{0,l-\kappa-ct\},l+\kappa+ct)} \mathcal{L}^1,
$$
where $\L^1$ denotes the  $1$-dimensional Lebesgue measure.
If $t_0=(l-\kappa)/c$ we get 
$$
D_x \z= D_l(t_0)+D_r(t_0)+ 2 c \delta(0)
$$
 (being $\delta$ the Dirac measure). The extra term comes from the 
 fact that $\a(w,w_x)(t_0,0^-)=-c$ and $\a(w,w_x)(t_0,0^+)=+c$. Similarly, when $t>t_0$ we get 
$$
D_x \z= D_l(t)+D_r(t)+ 2 c \theta(t) \delta(0),
$$ 
with $0<\theta(t)<c$ depending on the (finite) contact angle. 
 
 Having that information we track the proof of Proposition 2 in \cite{ARMAsupp} and we learn that our result would be proved if we were able to show that
$$
\int_{t_0}^T \int_{-l-\kappa-ct}^{l+\kappa+ct} \phi(t) w_t T(w) S(w) \ dxdt
\ge \int_{t_0}^T \int_{-l-\kappa-ct}^{l+\kappa+ct} D_x\a(w,w_x) \phi(t) T(w) S(w) \ dt ,
$$
for any $0 \le \phi \in \mathcal{D}((t_0,T)\times \R) $ and any $T \in \mathcal{T}^+,\, S \in \mathcal{T}^-$. In fact, Step 2 in the proof of Proposition 2 in \cite{ARMAsupp} already shows that we have 
$$
\int_{t_0}^T \int_{-l-\kappa-ct}^{l+\kappa+ct} \phi(t) w_t T(w) S(w) \ dxdt
\ge \int_{t_0}^T \int_{-l-\kappa-ct}^{l+\kappa+ct} D_x^{ac}\a(w,w_x) \phi(t) T(w) S(w) \ dt.
$$
As 
\begin{eqnarray*}
&&\hspace{-1cm}\int_{t_0}^T \int_{-l-\kappa-ct}^{l+\kappa+ct} D_x^s\a(w,w_x) \phi(t) T(w) S(w) \ dt 
\\&&= 2 c \int_{t_0}^T \theta(t) \phi(t,x=0) T(w(t,0)) S(w(t,0))\ dt \le 0
\end{eqnarray*}
in our particular case, the proof is complete.
\end{proof}
\begin{corollary}
\label{bebimocor}
Let $u_0\in BV(\R)$ with connected support and let $u$ be the associated entropy solution of \eqref{problema}. The following statements hold true:
\begin{enumerate}
\item Assume that $u_0$ is continuous at $x_0\in \mbox{int}\, (\supp u_0)$ and $u_0(x_0)=0$. Assume further that $u_0(x)>0$ for $x\in \mbox{int}\, (\supp u_0) \backslash \{x_0\}$. Then $u(t,x)>0$, for every $x\in \mbox{int}\, (\supp u(t))$ and every $t>0$.\label{primo}
\item Assume that $u_0(x)>0$ for $x\in \mbox{int}\, (\supp u_0) \backslash \{x_0\}$, being $x_0$ such that $u(x_0^-)=0$ and $u(x_0^+)>0$ (resp. $u(x_0^-)>0$ and $u(x_0^+)=0$). Then $u(t,x)>0$, for every $x\in \mbox{int}\, (\supp u(t))$ and every $t>0$. 
\label{bebimo}
\item A similar statement holds true if we substitute $x_0$ by a finite collection of points falling into any combination of cases (\ref{primo}) or (\ref{bebimo}). 
\end{enumerate}
\end{corollary}
\begin{proof}
There is no loss of generality in assuming that $x_0=0$. To prove the first point, let $\eps>0$ be given. Then we can find suitable parameters so that the profile $w$ constructed in Proposition \ref{wprofile} verifies that $w(0,x)\le u_0(x)$ and $w(\eps,x_0)>0$. As we can do this for any value of $\eps>0$, we deduce that $u(t,x_0)>0$ for any $t>0$. The rest is a consequence of Theorem \ref{th2}.

The proof of the second point is similar: We are able to find parameters such that $w(0,x)\le u_0(x)$ and $w(\eps,x_0)>0$, thus $u(t,x_0^-)>0$ and $u(t,x_0^+)>0$ for any $t>0$. Finally, the last statement is a consequence of the local character of Theorem \ref{th2} and Proposition \ref{wprofile}.
\end{proof}
Provided that that $u_t(t)$ is a finite Radon measure for any $t>0$, this result would show that any initial datum falling under points (\ref{primo}) or (\ref{bebimo}) of the previous results falls immediately under the assumptions of either Theorem \ref{ilustracion} or Proposition \ref{nocusp} (or any suitable modification of those with continuous interfaces). Hence the associated solution becomes eventually smooth.


\section{Smoothing effects for the relativistic heat equation: The general situation}
\label{generaliza}

Let us discuss now what happens when we consider an initial condition with a finite number of jump discontinuities. Keep in mind that a jump discontinuity could evolve into a point of continuity which is not Lipschitz, and that a point where $u_0$ vanishes could evolve into a point of continuity which is not Lipschitz and also into a jump discontinuity. From the point of view of our analysis in  Section \ref{caseIso}, the common trait that these singular points share is the fact that they allow no mass flux through them as long as they stand --the only noticeable difference is that zeros of $u_0$ disappear instantaneously, while non-Lipschitz continuity points and jump discontinuities may take some time to dissolve. 

We have discussed in Section \ref{caseIso} what would be the dynamics of an isolated singular point: It will eventually disappear. This will be also the case if we have an array of singular points initially, as long as the trajectories that they will trace out during evolution do not cross or do not meet those of the interfaces. In such a case we would be able to treat them one by one as isolated singular points. If this is true, the analysis of the evolution would be reduced to label and track carefully each trajectory traced out by a singular point as long as it is not dissolved. The following statement gives shape to these ideas.
\begin{proposition}
\label{general1}
Let $0\le u_0 \in BV(\R)$. Assume that $u_0$ is supported in $[a,b]$. Consider $\mathcal{S}_{u_0}=\{s_i\}\subset [a,b]$ a finite set, in which each of the $s_i$ is one of the following:
\begin{itemize}
\item a point in which $u_0$ has a jump discontinuity,
\item a point in which $u_0$ is continuous but not Lipschitz-continuous, 
\item a point in which $u_0$ has a zero.
\end{itemize}
 Assume also that $u_0 \in (W_{loc}^{1,\infty}\cap W^{1,1})(\R\backslash \mathcal{S}_{u_0})$ and $u_0(x)>0$ in $(a,b)\backslash \mathcal{S}_{u_0}$. Assume finally that $u_t(t)$ is a finite Radon measure for any $t>0$. Then:
\begin{enumerate}
\item $u(t) \in BV(\R)$ for each $t>0$. 
\item $u(t)$ is supported on $[a-ct,b+ct]$.
\item There exists some $0<T^*<\infty$ such that $u(t)\in W^{1,1}(a-ct,b+ct)$ and $u(t)$ is smooth inside its support, for every $t\ge T^*$. Moreover, $u(t)>0	, \forall x \in (a-ct,b+ct)$.
\end{enumerate}
\end{proposition}
\begin{proof}
To start with, we notice that no singularity overlap can take place during the dynamical evolution, due to the fact that mass flux is not allowed through any such singular point. If the trajectories traced out by two singular points happen to cross, a Dirac measure would appear at the crossing location due to mass preservation. But this is not possible, since $u_0 \in L^\infty(\R)$. Now, based on our previous results, there is some $t_1>0$ so that the cardinal of the set $S_v(t)$ is constant for every $0<t<t_1$. Choose $t_1$ to be maximal with this property. Then we define $\mathcal{S}_{ess}(u_0)=\varphi(0,S_v(t_1/2))$. This is the set of points that are associated with singularities that are not dissolved instantaneously, the only ones we have to worry about. In fact, as a consequence of the results in Section \ref{caseIso}, members of $\mathcal{S}_{ess}(u_0)$ fall at most into one of two categories: Jump discontinuities or points of continuity such that both lateral traces of $\b$ happen to be $+c$ or $-c$.

Say that $\mathcal{S}_{ess}(u_0)=\{p_i\}_i,\, i=1,\ldots,n$. Let us term $P_i(t)=(p_i\pm ct,p_{i+1}\pm ct):=(p_i(t),p_{i+1}(t)),\ i=1,\ldots,n-1$ the corresponding virtual evolution of the connected components of $[a,b]\backslash J_{u_0}$ for $t>0$. We choose $\pm$ according to Rankine--Hugoniot relations when we are tracking a jump discontinuity. When dealing with points of continuity in which Lipschitz continuity does not hold, we choose $``+$'' if both lateral traces of $\b$ are $-c$ and $``-$'' if both lateral traces of $\b$ are $+c$ (note that this can be regarded as a limiting case of Rankine--Hugoniot relations). Now we may define
$$
m_i = \int_{P_i(0)}u_0\ dx >0,\ i=1,\ldots, n-1.
$$
Then, since none of the trajectories given by $p_i(t)$ cross, we have that
$$
\int_{P_i(t)} u(t)\ dx = m_i,\ i=1,\ldots, n-1
$$
as long as no singularity is dissolved. Thus, what we do is to consider the set of maps
$$
  \varphi_i(t,m)= p_i(t) + \int_0^m v^i(t,r)\ dr, \quad u(t,\varphi_i(t,m)) = \frac{1}{v^i(t,m)},\quad i=1,\ldots,n,
$$ 
which define a set of functions $v^i:(t_1/2,t_1)\times (0,m_i)\rightarrow \R^+,\ i=1,\ldots,n-1$. Each function $v^i$ falls under the hypothesis of Theorem \ref{extendedMDF} --and moreover $S_{v^i}(t_1/2)=\emptyset$. In that way we get a description of the evolution of $u(t)$ in terms of the functions $v^i(t)$ as long as there is no breakdown of singularities. 

Thus, there is a first time $t_1$ for which a singularity (meaning a jump discontinuity or a continuity point at which Lipschitz continuity does not hold) will be dissolved, say that at $p_2(t_1)$. Then what we do is to merge $P_1(t_1)$ and $P_2(t_1)$ into one single component $\tilde{P}_1(t),\ t\ge t_1$ enclosing a quantity of mass $\tilde{m}_1:=m_1+m_2$, while we relabel the remaining $P_i(t_1)$ accordingly and reset the inverse distribution formulation for each $\tilde{P}_i(t),\ t\ge t_1$ in terms of a reduced set of functions $v^i,\ i=1,\ldots,n-2$. We modify this procedure accordingly if two or more singularities happen to vanish at the same time. This new description can be used until another singularity vanishes at a time $t_2$, on which we repeat the relabeling operation and we reset again the inverse distribution formulation for each separated piece. We can continue in this fashion until every singularity which was initially present has vanished, which is the case thanks to the results in Section \ref{caseIso}. 
\end{proof}

The previous result covers the case of a connected compact support. Let us address now the general case:
\begin{theorem}
\label{Roma}
Let $0\le u_0 \in BV(\R)$ and let $\supp u_0$ be a disjoint union of closed intervals. Consider $\mathcal{S}_{u_0}=\{s_i\}\subset \supp u_0$ such that $\mathcal{S}_{u_0}$ is finite on each connected component of $\supp u_0$, in which each of the $s_i$ is one of the following:
\begin{itemize}
\item a point in which $u_0$ has a jump discontinuity,
\item a point in which $u_0$ is continuous but not Lipschitz-continuous,
\item a point in which $u_0$ has a zero.
\end{itemize}
Let $\mbox{int}$ denote the interior of a set; assume also that $u_0 \in (W_{loc}^{1,\infty}\cap W^{1,1})(\R\backslash \mathcal{S}_{u_0})$ and $u_0(x)>0$ for every $x \in \mbox{int}\, (\supp u_0) \backslash \mathcal{S}_{u_0}$. Assume finally that $u_t(t)$ is a finite Radon measure for any $t>0$. Then:
\begin{enumerate}
\item $u(t) \in BV(\R)$ for each $t>0$. 
\item There exists some $0<T^*<\infty$ such that $u(t)\in W^{1,1}(\mbox{int}\, (\supp u(t)))$ and $u(t)$ is smooth inside its support, for every $t\ge T^*$. Moreover, $u(t)>0$ in the support.
\end{enumerate}
\end{theorem}
\begin{proof}
It is mostly straightforward: We apply Proposition \ref{general1} to each connected component in the initial support. In fact, an obvious modification of Proposition \ref{general1} applies to the case of a connected support which is not compact (if any such component is present): The support is no longer $[a-ct,b+ct]$ and we have to replace it with $\supp u_0 \oplus B(0,ct)$. Once this is done, the result applies mutatis mutandis. Thus, this procedure describes what happens as long as no pair of connected components interact. When two (or more) connected components meet, we consider the union of them as a new connected component of the support. At the merging time $t=t_m$, the solution may have a singularity at each contact point, depending on how the meeting interfaces were. (More specifically we may get a jump discontinuity, a continuous zero --maybe not being Lipschitz continuous-- or a continuity point of strict positivity --where we may lack of Lipschitz regularity.) These are all instances that we met previously, so we consider the solution at $t=t_m$ as a new initial datum and we apply Proposition \ref{general1} --more precisely a variant of it allowing for unbounded supports-- to each of the connected components. We repeat the procedure until no more connected components will merge (which is a finite time that we can estimate in terms of the initial configuration of connected components) and in this way the result is proved.
\end{proof}

\section{Regularity for the FLPME before contact time}
We can state a local regularity result:
\begin{proposition}
Let
$u_0\in BV(\R)$ with $u_0(x) \geq \kappa > 0$ for $x\in [a,b]$, and $u_0(x)=0$
for $x\not\in [a,b]$. Assume that $u_0$ is locally Lipschitz in its support out of a finite set $\varphi(0,S_v(0))$. Let $T^*$ be defined by Corollary \ref{finitecontact}. Then the entropy solution $u$ of \eqref{parab} 
satisfies the following additional properties:
\begin{itemize}

\item  $u(t)\in W^{1,\infty}_{loc}((a(t),b(t))\backslash \varphi(t,S_v(t)))$ for every $t\in (0,T^*)$.
\item $u(t)$ is smooth in $(a(t),b(t))\backslash \varphi(t,S_v(t))$ for $t<T^*$ (in fact $u$ is smooth in \newline
$\cup_{0<t<T^*}(\{t\}\times ((a(t),b(t))\backslash \varphi(t,S_v(t))))$).
\item $u(t)\in BV(\R)$ for every $t\in (0,T^*)$. Moreover, if $u_0 \in W^{1,1}(0,M)$ then $u(t)\in W^{1,1}(\R)$ for every $t\in (0,T^*)$.
\end{itemize}
\end{proposition}

Roughly speaking, this result shows that, up to the time in which (at least one) interfaces become continuous, the solution undergoes some regularizing effect. In fact, Lipschitz cusps are regularized instantaneously, while no new jump discontinuities and/or points with H\"older continuity appear. This can be shown by a careful adaptation of the arguments in the preceding Section \ref{rhestatements}. Moreover, arguing as in Lemma \ref{handy} H\"older cusps vanish instantaneously, and arguing as Proposition \ref{p13} the size of any jump discontinuity does not increase. The main technical difficulty that we face in order to try to extend this result beyond $T^*$ is that we do not know how to make sense of the inverse distribution formulation in that case. Were this possible, then the arguments in Sections \ref{caseIso} and \ref{generaliza} would likely imply a complete smoothing effect on the long time run as the one in Theorem \ref{Roma} (replacing Proposition \ref{concdecay2} by Proposition \ref{pr31} this time). 

{Regarding the case of initial data with continuous interfaces, local-in-time regularity results were shown in \cite{CC} for initial data having global Lipschitz regularity. The local character of these results, together with heuristic arguments  and numerical simulations like that in Fig. \ref{fi4} and those in \cite{CKR} suggest that there will be a loss of regularity which is connected with a waiting time phenomenon. Anyhow, after the support starts to spread we expect smoothing effects to operate on the solution.}

\section{Appendix: Entropy solutions}
\label{sc11}
We collect here below some definitions that are needed to
work with entropy solutions of flux limited (or saturated) diffusion equations.

Note that both equation (\ref{problema}) and \eqref{parab} can be written as
\begin{equation} \label{DirichletproblemP}
u_t = \a(u,  u_x)_x, \qquad \hbox{in \hspace{0.2cm} $Q_T=(0,T)\times \R$}
\end{equation}
where $\a(z,\xi) = \nabla_\xi f(z,\xi)$ and (setting $\nu=c=1$)
\begin{equation}\label{funct:frhe}
f(z,\xi) = z  \sqrt{z^2 + \vert \xi\vert^2} \quad (\mbox{resp.}\ f(z,\xi) = \frac{1}{m} z^{2-m} \sqrt{1+m^2z^{2-2m} |\xi|^2}).
\end{equation}
As usual, we define
\begin{equation}\label{funct:arhe}
h(z,\xi) = \a(z,\xi)\cdot \xi = \frac{z \vert
\xi\vert^2}{\sqrt{z^2 + \vert \xi\vert^2}} \quad (\mbox{resp.}\ h(z,\xi)= \frac{m z^m |\xi|^2}{\sqrt{1+m^2z^{2-2m} |\xi|^2}}).
\end{equation}
We also let $\b(z,\xi)$ be defined by 
\begin{equation}
\nonumber
\a(z,\xi)=z \b(z,\xi).
\end{equation} 
Note that $f$ is convex in $\xi$ and both $f,h$ have linear growth as $|\xi|\to\infty$.

\subsection{Functions of bounded variation and some generalizations}\label{sect:bv}
Denote by ${\mathcal L}^N$ and ${\mathcal H}^{N-1}$ the
$N$-dimen\-sio\-nal Lebesgue measure and the $(N-1)$-dimen\-sio\-nal
Hausdorff measure in $\R^N$, respectively. Given an open set
$\Omega$ in $\R^N$  we denote by ${\mathcal D}(\Omega)$  the space
of infinitely differentiable functions with compact support in
$\Omega$. The space of continuous functions with compact support
in $\R^N$ will be denoted by $C_c(\R^N)$.

Recall that if $\Omega$ is an open subset of $\R^N$, a function $u
\in L^1(\Omega)$ whose gradient $Du$ in the sense of distributions
is a vector valued Radon measure with finite total variation in
$\Omega$ is called a {\it function of bounded variation}. The
class of such functions will be denoted by $BV(\Omega)$.  For $u
\in BV(\Omega)$, the vector measure $Du$ decomposes into its
absolutely continuous and singular parts $Du = D^{ac} u + D^s u$.
Then $D^{ac} u = \nabla u \ \L^N$, where $\nabla u$ is the
Radon--Nikodym derivative of the measure $Du$ with respect to the
Lebesgue measure $\L^N$. We also split $D^su$ in two parts: the
{\it jump} part $D^j u$ and the {\it Cantor} part $D^c u$. We say that $x \in \Omega$ is an approximate jump point of $u$ if there exist $u^+(x) \neq u^-(x) \in \R$ and $\nu_u(x) \in  \mathbb{S}^{d-1}$ such that
$$
\lim_{\rho \searrow 0} \frac{1}{\L(B_\rho^+(x,\nu_u(x)))} \int_{B_\rho^+(x,\nu_u(x))} |u(y)-u^+(x)| \, dy = 0
$$
and
$$
\lim_{\rho \searrow 0} \frac{1}{\L(B_\rho^-(x,\nu_u(x)))} \int_{B_\rho^-(x,\nu_u(x))} |u(y)-u^-(x)| \, dy = 0,
$$
where
$$
B_\rho^+(x,\nu_u(x)) = \{y\in B(x,\rho)/ (y-x)\cdot \nu_u(x)>0\}
$$
and
$$
B_\rho^-(x,\nu_u(x)) = \{y\in B(x,\rho)/ (y-x)\cdot \nu_u(x)<0\}.
$$
We denote by $J_u$ the set of approximate jump points. It is
well known (see for instance \cite{Ambrosio}) that 
$$
D^j u = (u^+
- u^-) \nu_u \H^{d-1} \res J_u,
$$ 
with $\nu_u(x) =\frac{Du}{\vert D u \vert}(x)$, 
being $\frac{Du}{\vert D u \vert}$
the Radon--Nikodym derivative of $Du$ with respect to its total
variation $\vert D u \vert$. For further information concerning
functions of bounded variation we refer to \cite{Ambrosio}.

We need to consider the following truncation functions. For $a <
b$, let $T_{a,b}(r) := \max(\min(b,r),a)$, $ T_{a,b}^l =  T_{a,b}-l$. We denote
\cite{ACMEllipticFLDE,ACM2005,ARMAsupp}
\[
\begin{split}
\mathcal T_r & := \{ T_{a,b} \ : \ 0 < a < b \},  
\\
\mathcal{T}^+ & := \{ T_{a,b}^l \ : \ 0 < a < b ,\, l\in \R, \, T_{a,b}^l\geq 0 \},  
\\
\mathcal{T}^- & := \{ T_{a,b}^l \ : \ 0 < a < b ,\, l\in \R, \, T_{a,b}^l\leq 0 \}.  
\end{split}
\]

Given any function $w$ and $a,b\in\R$ we shall use the notation
$\{w\geq a\} = \{x\in \R^N: w(x)\geq a\}$, $\{a \leq w\leq b\} =
\{x\in \R^N: a \leq w(x)\leq b\}$, and similarly for the sets $\{w
> a\}$, $\{w \leq a\}$, $\{w < a\}$, etc.

We need to consider the following function space
$$
TBV_{\rm r}^+(\R^N):= \left\{ w \in L^1(\R^N)^+  \ :  \ \ T_{a,b}(w) - a \in BV(\R^N), \
\ \forall \ T_{a,b} \in \mathcal T_r \right\}.
$$
Using the chain rule for BV-functions (see for instance
\cite{Ambrosio}), one can give a sense to $\nabla u$ for a
function $u \in TBV^+(\R^N)$ as the unique function $v$ which
satisfies
\begin{equation*}\label{E1WRN}
\nabla T_{a,b}(u) = v \1_{\{a < u  < b\}} \ \ \ \ \ {\mathcal
L}^N-{\rm a.e.}, \ \ \forall \ T_{a,b} \in \mathcal{T}_r.
\end{equation*}
We refer to \cite{Ambrosio} for details.
\subsection{A generalized Green's formula}
Assume that $\Omega$ is an open bounded set of $\R^N$ with Lipschitz continuous boundary. Let $p\ge 1$ and $p'$ its dual exponent. Following \cite{Anzellotti1}, let us denote
\begin{equation}
\nonumber
X_{p}(\R^N) = \{ \z \in L^{\infty}(\Omega, \R^N):
\div(\z)\in L^p(\R^N) \}.
\end{equation}
If $\z \in X_{p}(\Omega)$ and $w \in BV(\Omega)\cap L^{p'}(\Omega) $,
we define the functional $(\z\cdot Dw): \mathcal{C}^{\infty}_{c}(\Omega)
\rightarrow \R$ by the formula
\begin{equation}
\nonumber
\langle (\z \cdot Dw),\varphi\rangle := - \int_{\Omega} w \, \varphi \,
\div(\z) \, dx - \int_{\Omega} w \, \z \cdot \nabla \varphi \, dx.
\end{equation}
Then $(\z \cdot Dw)$ is a Radon measure in $\Omega$ \cite{Anzellotti1}, and
\begin{equation}
\nonumber
\int_{\Omega} (\z \cdot Dw) = \int_{\Omega} \z \cdot \nabla w \, dx, \ \
\ \ \ \forall \ w \in W^{1,1}(\Omega) \cap L^{\infty}(\Omega).
\end{equation}
Moreover, $(\z \cdot Dw)$ is
absolutely continuous with respect to $\vert Dw \vert$ \cite{Anzellotti1}.

In the case where the distribution $(\z \cdot Dw)$ is a Radon measure we denote by
$(\z \cdot Dw)^{ac}$, $(\z \cdot Dw)^{s}$ its absolutely continuous and singular parts with respect to $\mathcal{L}^d$.
One has that $(\z \cdot Dw)^{s}$ is absolutely continuous with respect to
$D^{s}w$ and $(\z \cdot Dw)^{ac}=\z\cdot \nabla w$.

The weak trace on $\partial \Omega$ of the normal component of $\z \in X_p(\Omega)$ is defined in \cite{Anzellotti1}. More precisely, it is proved that there exists a linear operator $\gamma : X_p(\Omega)\rightarrow L^\infty (\partial \Omega) $ such that $\|\gamma(\z)\|_\infty\le \|\z\|_\infty$ and $\gamma(\z)(x) = \z(x) \cdot \nu^\Omega(x)$ for all $x \in \partial \Omega$ --being $\nu^\Omega(x)$ the normal vector at $x$ which points outwards--, provided that $\z \in C^1(\bar \Omega,\R^N)$. We shall denote $\gamma(\z)(x)$ by  $[\z \cdot \nu^\Omega](x)$. Moreover, the following Green's formula, relating the function $[\z\cdot \nu^\Omega]$ and the measure $(\z \cdot Dw)$, for $\z\in X_p(\Omega)$ and $w\in BV(\Omega)\cap L^{p'}(\Omega)$, is proved
$$
\int_\Omega w \, \div(\z)\, dx + \int_\Omega (\z \cdot Dw) = \int_{\partial \Omega} [\z \cdot \nu^\Omega]w \, d \mathcal{H}^{N-1}.
$$

\subsection{Functionals defined on BV}\label{sect:functionalcalculus}
In order to define the notion of entropy solutions of
(\ref{DirichletproblemP}) and give a characterization of them, we
need a functional calculus defined on functions whose truncations
are in $BV$.

Let $\Omega$ be an open subset of $\R^N$. Let $g: \Omega \times \R
\times \R^N \rightarrow [0, \infty[$ be a Borel function such that
\begin{equation*}\label{LGRWTH}
C(x) \vert \xi \vert - D(x) \leq g(x, z, \xi)  \leq M'(x) + M
\vert \xi \vert
\end{equation*}
for any $(x, z, \xi) \in \Omega \times \R \times \R^N$, $\vert
z\vert \leq R$, and any $R>0$, where $M$ is a positive constant
and  $C,D,M' \geq 0$ are bounded Borel functions which may depend
on $R$. Assume that $C,D,M' \in L^1(\Omega)$.

Following Dal Maso \cite{Dalmaso} we consider the
functional:
\begin{eqnarray*}\label{RelEnerg}
{\mathcal R}_g(u)&:=& \displaystyle\int_{\Omega} g(x,u(x), \nabla
u(x)) \, dx + \int_{\Omega} g^0 \left(x,
\tilde{u}(x),\frac{Du}{\vert D u \vert}(x) \right) \,  d\vert D^c
u \vert
\nonumber \\
&&+ \displaystyle\int_{J_u} \left(\int_{u_-(x)}^{u_+(x)}
g^0(x, s, \nu_u(x)) \, ds \right)\, d \H^{N-1}(x),
\end{eqnarray*}
for $u \in BV(\Omega) \cap L^\infty(\Omega)$, being $\tilde{u}$ 
the approximated limit of $u$ \cite{Ambrosio}. The recession
function $g^0$ of $g$ is defined by
\begin{equation*}\label{Asimptfunct}
 g^0(x, z, \xi) = \lim_{t \to 0^+} tg \left(x, z, \frac{\xi}{t}
 \right).
\end{equation*}
It is convex and homogeneous of degree $1$ in $\xi$.

In case that $\Omega$ is a bounded set, and under standard
continuity and coercivity assumptions,  Dal Maso proved in
\cite{Dalmaso} that ${\mathcal R}_g(u)$ is $L^1$-lower semi-continuous
for $u \in BV(\Omega)$. More recently, De Cicco, Fusco, and Verde
\cite{DCFV} have obtained a very general result about the
$L^1$-lower semi-continuity of ${\mathcal R}_g$ in $BV(\R^N)$.

Assume that $g:\R\times \R^N \to [0, \infty[$ is a Borel function
such that
\begin{equation}\label{LGRWTHnox}
C \vert \xi \vert - D \leq g(z, \xi)  \leq M(1+ \vert \xi \vert)
\qquad \forall (z,\xi)\in \R^N, \, \vert z \vert \leq R,
\end{equation}
for any $R > 0$ and for some constants  $C,D,M \geq 0$ which may
depend on $R$. Observe that both functions $f,h$ defined in
(\ref{funct:frhe}), (\ref{funct:arhe}) satisfy (\ref{LGRWTHnox}).

Assume that
$$
\1_{\{u\leq a\}} \left(g(u(x), 0) - g(a, 0)\right), \1_{\{u \geq b\}} \left(g(u(x),
0) - g(b, 0) \right) \in L^1(\R^N),
$$
for any $u\in L^1(\R^N)^+$. Let $u \in TBV_{\rm r}^+(\R^N)  \cap
L^\infty(\R^N)$  and $T = T_{a,b}-l\in {\mathcal T}^+ $. For each
$\phi\in C_c(\R^N)$, $\phi \geq 0$, we define the Radon measure
$g(u, DT(u))$ by
\begin{eqnarray}\label{FUTab}
\langle g(u, DT(u)), \phi \rangle &: =& {\mathcal R}_{\phi g}(T_{a,b}(u))+
\displaystyle\int_{\{u \leq a\}} \phi(x)
\left( g(u(x), 0) - g(a, 0)\right) \, dx  \nonumber\\
&& \displaystyle  + \int_{\{u \geq b\}} \phi(x)
\left(g(u(x), 0) - g(b, 0) \right) \, dx.
\end{eqnarray}
If $\phi\in C_c(\R^N)$, we write $\phi = \phi^+ -
\phi^-$ with $\phi^+= \max(\phi,0)$, $\phi^- = - \min(\phi,0)$,
and we define $\langle g(u, DT(u)), \phi \rangle : =
\langle g(u, DT(u)), \phi^+ \rangle- \langle g(u, DT(u)), \phi^- \rangle$.

Recall that, if $g(z,\xi)$ is continuous in $(z,\xi)$, convex
in $\xi$ for any $z\in \R$, and $\phi \in C^1(\R^N)^+$ has compact
support, then  $\langle g(u, DT(u)), \phi \rangle$ is lower
semi-continuous in
\newline
 $TBV^+(\R^N)$ with respect to
$L^1(\R^N)$-convergence \cite{DCFV}. This property is used to prove existence of
solutions of (\ref{DirichletproblemP}).

We can now define the required functional calculus (see
\cite{ACM2005,ACMEllipticFLDE,leysalto}). Let us denote by ${\mathcal P}$ the set of Lipschitz continuous
functions $p : [0, +\infty[ \rightarrow \R$ satisfying
$p^{\prime}(s) = 0$ for $s$ large enough. We write ${\mathcal
P}^+:= \{ p \in {\mathcal P} \ : \ p \geq 0 \}$.

Let $S \in  \mathcal{P}^+$, $T \in \mathcal{T}^+$.
We assume that
$u \in TBV_{\rm r}^+(\R^N) \cap L^\infty(\R^N)$ and note that
$$
\1_{\{u\leq a\}} S(u)\left(f(u(x), 0) - f(a, 0)\right), \1_{\{u \geq b\}} S(u)\left(f(u(x),
0) - f(b,0) \right) \in L^1(\R^N).            
$$
Since $h(z, 0) = 0$, the last assumption clearly holds also for those 
$h$ defined in \eqref{funct:arhe}. We define by $f_S(u,DT(u))$, $h_S(u,DT(u))$ as the Radon
measures given by (\ref{FUTab}) with $f_{S}(z,\xi) = S(z)
f(z,\xi)$ and  $h_{S}(z,\xi) = S(z) h(z,\xi)$,
respectively.

\subsection{The  notion of of entropy solution}\label{sect:defESpp}
Let $L^1_{w}(0,T,BV(\R^N))$  be the space of weakly$^*$
measurable functions $w:[0,T] \to BV(\R^N)$ (i.e., $t \in [0,T]
\to \langle w(t),\phi \rangle$ is measurable for every $\phi$ in the predual
of $BV(\R^N)$) such that $\int_0^T \Vert w(t)\Vert_{BV} \, dt< \infty$.
Observe that, since $BV(\R^N)$ has a separable predual (see
\cite{Ambrosio}), it follows easily that the map $t \in [0,T]\to
\Vert w(t) \Vert_{BV}$ is measurable. By  $L^1_{loc, w}(0, T,
BV(\R^N))$ we denote the space of weakly$^*$ measurable functions
$w:[0,T] \to BV(\R^N)$ such that the map $t \in [0,T]\to \Vert
w(t) \Vert_{BV}$ is in $L^1_{loc}(]0, T[)$.

\begin{definition} \label{def:espb}
Assume that $u_0 \in (L^1(\R^N)\cap L^\infty(\R^N))^+$. A
measurable function\newline $u:  \,]0,T[\times \R^N \rightarrow \R$ is an
{\it entropy  solution}   of (\ref{DirichletproblemP}) in $Q_T =
]0,T[\times \R^N$ if \newline
$u \in C([0, T]; L^1(\R^N))$,
$T_{a,b}(u(\cdot)) - a \in L^1_{loc, w}(0, T, BV(\R^N))$ for all
$0 < a < b$, and
\begin{itemize}
\item[(i)]  $u(0) = u_0$, and \item[(ii)] \ the following
inequality is satisfied
\begin{eqnarray}\label{pei}
&& \hspace{-0.6cm}\displaystyle \int_0^T\int_{\R^N} \phi
h_{S}(u,DT(u)) \, dt + \int_0^T\int_{\R^N} \phi h_{T}(u,DS(u)) \, dt \nonumber
 \\ &&\hspace{-0.2cm}\leq  \displaystyle\int_0^T\int_{\R^N} \Big\{ J_{TS}(u(t)) \phi^{\prime}(t) - \a(u(t), \nabla u(t)) \cdot \nabla \phi \
T(u(t)) S(u(t))\Big\} dxdt,  \nonumber
\end{eqnarray}
 for truncation functions $S,  T \in \mathcal{T}^+$, and any  smooth function $\phi$ of
 compact support, in particular  those  of the form $\phi(t,x) =
 \phi_1(t)\rho(x)$, $\phi_1\in {\mathcal D}(]0,T[)$, $\rho \in
 {\mathcal D}(\R^N)$, where $J_q(r)$ denotes the primitive of $q$ for any function $q$; i.e. $\displaystyle J_q(r):=\int_0^r q(s)\,ds$
\end{itemize}
\end{definition}
\begin{remark}
Due to the fact that $r \in [0,\infty) \mapsto r^m$ is a smooth and strictly increasing function, we can develop $\nabla u^m$ using the chain rule. Then the theory of non-negative entropy solutions to \eqref{parab} given in \cite{CCC} is the same as the theory of non-negative entropy solutions to
\begin{equation}
\label{alternateform}
\frac{\partial u}{\partial t} = \nu \, \mbox{div}\, \left( \frac{m u^m \nabla u}{\sqrt{1+\frac{\nu^2}{c^2}|m u^{m-1}\nabla u|^2}} \right).
\end{equation}
as given in \cite{ACMEllipticFLDE,ACM2005} and stated here. We refer to Remarks 3.3 and 3.8 in \cite{CCC} in that concern. Thus, in order to have a single framework, we use the theory in \cite{ACMEllipticFLDE,ACM2005,CEU2} to deal with \eqref{parab} and \eqref{problema}.
\end{remark}

\subsection{Well-posedness}

The following result states the well-posedness of the problems we are interested in. Besides, it provides us with a comparison principle for solutions of these.
\begin{theorem}
\label{WP}
For any initial datum $0 \le u_0 \in L^1(\R^N)\cap L^\infty(\R^N)$ there exists a unique entropy solution $u$ of \eqref{parab} (resp. \eqref{problema}) in $Q_T =(0,T)\times \R^N$ for every $T > 0$, such that $u(0) = u_0$. Moreover, if $u(t), \overline{u}(t)$ are the entropy solutions corresponding to initial data $u_0, \overline{u}_0 \in L^1(\R^N)^+$, respectively, then
\begin{equation}
\label{L1contraction}
\|(u(t) - \overline{u}(t))^+Ê\|_1 \le \|(u_0 -  \overline{u}_0)^+ \|_1 \quad \mbox{for all} \  t \ge 0.
\end{equation}
\end{theorem}
For a proof see \cite{ACM2005,CCC}. 
\begin{remark}\label{remark:bv}{\rm We observe that $u(t)\in BV(\R^N)$ for any $t>0$ if $u_0\in BV(\R^N)$.
Indeed, let $\tau_h u_0(x) = u_0(x+h)$, $h\in \R^N$. Let $u_h$ be the entropy solution corresponding to the initial datum
$\tau_h u_0$. Then by the uniqueness result of Theorem \ref{WP} we have that $u_h(t) = \tau_h u(t)$ for any $t \geq 0$.
By applying estimate  (\ref{L1contraction}) we have
$$
\Vert u(t) - \tau_h u(t)\Vert_1 \leq \Vert u_0 - \tau_h u_0\Vert_1 \qquad \forall t> 0.
$$
Since $u_0 \in BV(\R^N)$ we deduce that $u(t)\in BV(\R^N)$ for all $t > 0$ and
$\Vert u(t)\Vert_{BV} \leq \Vert u_0\Vert_{BV}$. Clearly $u\in L^1_w (0,T;BV(\R^N))$.
}
\end{remark}

\subsection{Sub- and super-solutions}
We need to use an extension of the notion of sub- and super-solutions initially proposed in \cite{ARMAsupp}. The aforementioned extension was introduced in \cite{Giacomelli}.
\begin{definition}
\label{subsuper}
A measurable function $u : (0,T) \times \R^N \rightarrow \R_0^+$ is an entropy sub- (resp. super-) solution of \eqref{DirichletproblemP} in $Q_T$ if $u \in C([0,T],L^1(\R^N))$, $T_{a,b}^a(u) \in L_{loc,w}^1(0,T,BV_{loc}(\R^N))$ for every $0<a<b$, $\a(u,\nabla u)\in L_{loc}^\infty(\R^N)$ for a.e. $t\in (0,T)$ and the following inequality is satisfied:
\begin{equation}
\begin{array}{c}
\displaystyle \int_0^T\int_{\R^N} \phi
h_{S}(u,DT(u)) \, dt + \int_0^T\int_{\R^N} \phi h_{T}(u,DS(u)) \, dt
\\
\geq  \displaystyle\int_0^T\int_{\R^N} \Big\{ J_{TS}(u(t)) \phi^{\prime}(t) - \a(u(t), \nabla u(t)) \cdot \nabla \phi \
T(u(t)) S(u(t))\Big\} dxdt,  \label{seineq}
\end{array}
\end{equation}
(resp. with $\le$) 
for any $\phi \in \mathcal{D}(Q_T)^+$ and any truncations $T \in \mathcal{T}^+,\ S \in \mathcal{T}^-$.
\end{definition}
This implies that (resp. with $\ge$)
\begin{equation}
\label{sinside}
u_t \le \div \a(u,\nabla u) \quad \mbox{in}\ \mathcal{D}'(Q_T).
\end{equation}
\begin{theorem}
\label{theocomp}
Consider either \eqref{parab} or \eqref{problema}. Then:
\begin{enumerate}
\item[\cite{ARMAsupp}] Estimate \eqref{L1contraction} holds true when $u(t)$ is replaced by a sub-solution such that $u(t) \in BV(\R^N)$ a.e. $0<t<T$, or when $\overline{u}(t)$ is replaced by a super-solution such that $\overline{u}(t) \in BV(\R^N)$ a.e. $0<t<T$.

\item[\cite{Giacomelli}] Estimate \eqref{L1contraction} holds true when $\overline{u}(t)$ is replaced by a super-solution such that $\bar u(t) \in BV_{loc}(\R^N)$ for a.e. $t \in (0, T)$ and the solution $u$ satisfies that $\supp u \cap ([0, T] \times \R^N)$ is compact.
\end{enumerate}
\end{theorem}
A consequence of the above Theorem is the following statement \cite[Theorem 2.5]{Giacomelli}: Let $\bar u$ be an entropy super-solution to the Cauchy
problem  \eqref{problema} in $(0, T)$ with $\bar u(0) \in L^\infty_{loc}(\R^N)$. If $\bar u(0) \geq u(0)$, $\bar u(t) \in BV_{loc}(\R^N)$ for a.e. $t \in (0, T)$ and $\supp u \cap ([0, T] \times \R^N)$ is compact,
then $\bar u(t) \geq u(t)$ for all $t \in (0, T)$.

The following sub-solutions constructed in \cite{ARMAsupp} will be useful for our purposes: 
\begin{proposition}
\label{subsol}
Given $R_0,\alpha_0 >0$ and $\gamma_0 \ge 0$, there are values $\beta_1,\beta_2 >0$ large enough such that
$$
u(t,x)=\left\{
\begin{array}{lr}
\exp \{Ê-\beta_1 t - \beta_2 t^2\} \left(\alpha_0 \frac{c}{\nu}\sqrt{(R_0 + ct)^2-|x|^2}
 +\gamma_0 \right) & \mbox{if}\ |x| < R_0 + ct
\\ \\
0 & \mbox{if}\ |x| \ge R_0 + ct
\end{array}
\right.
$$
is an entropy sub-solution of \eqref{problema}.
\end{proposition}
This sub-solution was used in \cite{ARMAsupp} to prove a number of qualitative properties for \eqref{problema}. We mention the following:
\begin{theorem}
\label{th2}
Let $u_0\in (L^1(\R)\cap L^\infty(\R))^+$ and let $u$ be the associated entropy solution of \eqref{problema}. The following assertions hold true:
\begin{enumerate}
\item Let $C\subset \R$ be open and bounded. Assume that $\mbox{supp}\ u_0 = \overline{C}$ and that for any closed set $F\subset C$ there is some $\alpha_F>0$ such that $u_0 \ge \alpha_F$ in $F$. Then
$$
\mbox{supp}\ u(t) = \overline{C} \oplus B(0,ct)\quad \forall t >0.
$$
\item Let $x \in \mbox{supp}\ u_0$ such that $u_0(y)\ge \alpha >0$ for any $y \in B(x,R),\, R>0$. Then $u(t,y)\ge \alpha (t)$ for any $y \in B(x,R+ct)$ and any $t>0$, for some positive function $t \mapsto \alpha(t)$. 
\end{enumerate}
\end{theorem}

A specific family of super-solutions for \eqref{problema} was constructed in \cite{Carrillo} to deal with continuous interfaces. Here is their result.
\begin{lemma}
\label{supercont}
Let $N=1$ and $U(t,x) = A(t) ((R_0+ct)^2-x^2)^\alpha$, $\alpha >0$. If $A'(t) \ge 0$, then $U(t,x)$ is a super-solution of \eqref{problema}.
\end{lemma}

The following result is a direct consequence of the results in \cite{Giacomelli,mine}:
\begin{proposition}
Let $u_0\in L^1(\R^N) \cap L^\infty(\R^N)^+$ be compactly supported. Then the entropy solution $u$ of \eqref{parab} launched by $u_0$ is compactly supported for each $t>0$ and the support of $u(t)$ spreads no faster than $c$. 
\end{proposition}

\subsection{Log-concave solutions of \eqref{problema} and decay of the sup norm}

\begin{definition}
Let $u : \Omega \rightarrow \R^+$. We say that $u$ is log-concave if $\log (u) :\Omega \rightarrow \R$ is a concave function.
\end{definition}

\begin{lemma}
Let $u: \Omega \rightarrow \R^+$ be a log-concave function.  Then $u \in W_{loc}^{1,\infty}(\Omega)$ and $\nabla u \in BV(\Omega)$. Moreover, $\|Ê\nabla u\|_\infty$ attains its maximum at $\partial \Omega$.
\end{lemma}
The following is a particular instance of a more general result stated in \cite{Andreu2008}.
\begin{theorem}
\label{logc}
Consider $u_0 \in L^1(\R)\cap L^\infty(\R)$ such that $u_0(x) \ge \alpha >0$ for $x\in \Omega$ and $u_0=0$ outside $\Omega$, which is assumed to be open, connected and bounded. Assume further that $u_0 \in W^{2,1}(\Omega),\ (u_0)_x \in L^\infty(\Omega)$ and also that $u_0$ is log-concave in $\bar{\Omega}$. Let $u(t,x)$ be the entropy solution of \eqref{problema} in $(0,T)$ with $u(0,x) = u_0(x)$. Then the following hold:
\begin{enumerate}
\item $u(t)$ is log-concave in $\Omega \oplus B(0,ct)$,
\item $u \in BV((0,T) \times \R)$; in fact, $u_t(t)$ is a Radon measure in $\R$ for each $t>0$,
\item $u$ is $C^\infty$ in $\Omega_T :=\{(t,x): t\in(0,T),x\in \Omega \oplus B(0,ct)\}$,
\item $u(t,x)=0$ for any $t\in(0,T)$ and $x \in \R \backslash (\Omega \oplus B(0,ct))$,
\item $u(t)$ has a vertical contact angle at the boundary of $\Omega \oplus B(0,ct)$ for almost all $t \in (0, T )$.
\end{enumerate}
\end{theorem}

These particular solutions are helpful in order to get some control over $\|u(t)\|_\infty$. We can prove the following statement.

\begin{proposition}
\label{concdecay2}
Let $u_0$ be an even, compactly supported initial condition which is non-negative, bounded and log-concave. Let $[-R,R]$ be its support and assume that $u_0(x) \ge \alpha >0$ for $x\in (-R,R)$. Let $\mathcal{SC}(r)=[-r-ct,r+ct]$ be the sound cone about $[-r,r]$. Then, for any $\eps \in (0,R)$ the associated entropy solution of \eqref{problema} verifies that
$$
\|u(t) \|_{L^\infty(\mathcal{SC}(R)\backslash \mathcal{SC}(\eps))}\le \frac{M}{2(\eps+ct)} \ \forall t\ge 0.
$$
\end{proposition}
\begin{proof}
This is just the combination of mass conservation, log-concavity and symmetry. All the former are preserved during evolution. The point is that any log-concave profile which is even is decreasing outwards. Thus, a geometric argument shows  that $M\ge 2 u(t,\eps+ct)(\eps+ct)$ and the result follows.
\end{proof}
\begin{remark}
We get the following estimate in arbitrary dimension: 
$$
\|u(t) \|_{L^\infty(\mathcal{SC}(R)\backslash \mathcal{SC}(\eps))}\le \frac{M}{|\mathbb{S}^{N-1}|(\eps+ct)^N} \ \forall t\ge 0.
$$ 
\end{remark}
Using a comparison argument, the decay estimate in Proposition \ref{concdecay2} holds true (after a suitable spatial translation) for initial data $u_0\in (L^1(\R)\cap L^\infty(\R))^+$ which are compactly supported.

\subsection{Rankine--Hugoniot conditions}

The results we quote next are particular instances of those in \cite{leysalto,CCC}.
\begin{proposition}
\label{p4}
Let $u \in C([0,T];L^1(\R))$ be the entropy solution of \eqref{problema} (resp. \eqref{parab}) with $u_0 \in BV(\R)^+$. Assume that $u \in BV_{loc}((0,T)\times \R)$. Then the speed of any discontinuity front is precisely $c$. Moreover, there holds that for almost any $t\in (0,T)$
\begin{equation}
\label{slopeconfig}
[\z \cdot \nu^{J_{u(t)}}]_+ = u^+\quad \mbox{and}\quad [\z \cdot \nu^{J_{u(t)}}]_- = u^-
\end{equation}
on each point of $J_{u(t)}$, being $[\z \cdot \nu^{J_{u(t)}}]_+$ and $[\z \cdot \nu^{J_{u(t)}}]_-$ the lateral traces of the flux.
\end{proposition}

A sufficient condition granting the required time regularity is the following particularization to the one-dimensional case of Proposition 4.2 in \cite{leysalto} and Proposition 6.2 in \cite{CCC}:
\begin{lemma}
\label{l3}
Let $u_0 \in BV(\R)^+$ satisfy the following:
\begin{enumerate}
\item $J_{u_0}$ is a finite set.
\item $u_0$ is either zero or bounded away from zero in any connected component of $\R\backslash J_{u_0}$.
\item $u_0 \in W^{2,1}(\R\backslash J_{u_0})$ and $(u_0)_x \in L^\infty (\R\backslash J_{u_0})$.
\item Given $x \in J_{u_0}$ and choosing $\nu^x=+1$, then $u_x^-,u_x^+\ge 0$ if $u^-<u^+$ and $u_x^-,u_x^+\le 0$ if $u^->u^+$.
\end{enumerate}
Let $u$ be the entropy solution of \eqref{parab} (resp. \eqref{problema}). Then $u_t(t)$ is a finite Radon measure in $\R$ for any $t>0$. As a consequence, $u\in BV([\tau,T]\times \R)$ for any $\tau>0$.
\end{lemma}
\begin{remark}
It is conjectured in \cite{leysalto,CCC} that $u_0 \in BV(\R^N)$ should suffice in order to have $u\in BV([\tau,T]\times \R^N)$ for any $\tau>0$.
\end{remark}
\begin{remark}
\label{r2}
If we work in one spatial dimension, whenever we are able to ensure that $u_t(t)$ is a finite Radon measure in $\R$ for any $t>0$, we have that 
$$
\frac{u u_x}{\sqrt{u^2 + \frac{\nu^2}{c^2}(u_x)^2}} \in BV(\R).
$$
A particular consequence is that we can deal with boundary traces of the above ratio.
\end{remark} 


\end{document}